\RequirePackage{fix-cm}
\documentclass[smallextended]{svjour3}       % onecolumn (second format)
\smartqed  % flush right qed marks, e.g. at end of proof
\usepackage{amsmath}
\usepackage{graphicx}
\usepackage{amssymb}
\usepackage{multirow}
\usepackage{comment}
\usepackage{csquotes}
\usepackage{color}
\usepackage{subfigure}
\usepackage{algorithm}  
\usepackage{algpseudocode}

\makeatletter
\def\cl@chapter{}  

\makeatother
\usepackage{cleveref}  

\begin{document}

\title{Adaptive Multi-Grade Deep Learning for Highly Oscillatory Fredholm Integral Equations of the Second Kind %\thanks{Grants or other notes
}

\titlerunning{Adaptive Deep Learning}        % if too long for running head

\author{Jie Jiang        \and
        Yuesheng Xu %etc.
}


\institute{Jie Jiang  \at
              School of Computer Science and Engineering, Sun Yat-sen University,Guangzhou 510275, People's Republic of China \\
              \email{jiangj73@mail2.sysu.edu.cn}           %  \\
           \and
           Yuesheng Xu \at
              Department of Mathematics and Statistics, Old Dominion University,Norfolk, VA 23529, USA \\
              \email{y1xu@odu.edu} 
}

\date{Received: date / Accepted: date}

\maketitle

\begin{abstract}
This paper studies the use of Multi-Grade Deep Learning (MGDL) for solving highly oscillatory Fredholm integral equations of the second kind. We provide rigorous error analyses of continuous and discrete MGDL models, showing that the discrete model retains the convergence and stability of its continuous counterpart under sufficiently small quadrature error. We identify the DNN training error as the primary source of approximation error, motivating a novel adaptive MGDL algorithm that selects the network grade based on training performance. Numerical experiments with highly oscillatory (including wavenumber 500) and singular solutions confirm the accuracy, effectiveness and robustness of the proposed approach.
\end{abstract}

\section{Introduction}

This paper investigates the use of deep neural networks (DNNs) for the numerical solution of the oscillatory Fredholm integral equation of the second kind. These equations arise in a variety of physical and engineering applications, notably in wave propagation and scattering problems such as electromagnetic scattering \cite{2010The,2003Inverse}. While classical numerical methods for such equations are well established \cite{Atkinson,2010The,2015Multiscale,wang2015oscillation,Xiang2023}, recent advances in scientific machine learning have spurred interest in exploring DNN-based methods. DNNs have shown promising results for solving partial differential equations (PDEs) \cite{raissi2018deep,raissi2019physics}, variational problems \cite{2018The}, and more recently, integral equations \cite{jiang2024deep}.

A major challenge in applying DNNs to oscillatory integral equations is the high-frequency nature of the solution, which inherits oscillations from the kernel \cite{Brunner2015,wang2015oscillation}. Standard DNNs exhibit a well-documented \emph{spectral bias} \cite{Rahaman2019}, meaning they tend to learn low-frequency components first and struggle to capture high-frequency features. This bias makes it particularly difficult to approximate solutions dominated by high-frequency modes, as is typical in oscillatory integral equations with complex exponential kernels.

To address this limitation, the \emph{multi-grade deep learning} (MGDL) framework was introduced in \cite{Xu:2023aa}. Rather than training a deep network in an end-to-end fashion, MGDL incrementally constructs the network grade-by-grade through a sequence of shallow subnetworks. Each grade builds upon the previous approximation by taking its output as input and learning to approximate the residual. 
Early grades capture coarse, low-frequency components, while later grades are exposed to finer, high-frequency residuals—effectively forcing the network to learn oscillatory features that standard DNNs tend to neglect due to spectral bias. This residual learning strategy naturally decouples the frequency content across grades, allowing each subnetwork to specialize in a different spectral band. As a result, MGDL adapts well to the multiscale nature of oscillatory solutions and has demonstrated effectiveness in function approximation \cite{Xu2023(2),Xu:2023aa}, partial differential equations \cite{XuZeng2023}, and more recently in countering spectral bias in deep learning \cite{fang2024addressing}.

In this paper, we extend the MGDL framework to oscillatory Fredholm integral equations by developing a novel adaptive algorithm that automatically determines the number of grades based on the training error. While MGDL was previously applied empirically to such equations in \cite{jiang2024deep}, no theoretical foundation was established. Here, we provide a rigorous mathematical analysis of both the continuous and discrete MGDL models in this setting.

A major limitation of standard deep learning methods is the need to pre-specify network depth before training. Since the optimal depth is typically unknown, suboptimal performance often necessitates adjusting the depth and retraining the network from scratch. In contrast, MGDL builds networks incrementally, grade by grade, eliminating the need to predefine depth. This incremental construction naturally leads to the adaptive MGDL (AMGDL) algorithm, a novel approach not studied in previous work, for which we establish rigorous theoretical guarantees.

Our main contributions are as follows:
\begin{itemize}
    \item We provide a rigorous error analysis of both continuous and discrete MGDL models, showing that the discrete model retains the convergence and stability properties of its continuous counterpart under small integration error.
    \item We identify the DNN training error as the dominant source of approximation error, motivating an adaptive approach based on training performance.
    \item We propose the AMGDL algorithm that dynamically selects the network grade to balance model complexity and solution accuracy.
    \item We validate the effectiveness of the adaptive strategy through numerical experiments involving highly oscillatory and singular solutions.
\end{itemize}

The remainder of the paper is organized as follows. Section~2 introduces the oscillatory Fredholm integral equation and reviews related DNN-based methods, including both the continuous and discrete MGDL models. Sections~3 and 4 present error analysis for the continuous and discrete MGDL models, respectively. Section~5 introduces the AMGDL algorithm and establishes its theoretical guarantees. Section~6 presents numerical experiments that demonstrate the accuracy and robustness of AMGDL, even for problems with wavenumber 500 and singularities, and confirm the theoretical results established in this paper. The paper concludes in Section~7.

\section{Deep Learning Model}
\label{sec_single}

In this section, we present the oscillatory Fredholm integral equation under consideration and describe the deep learning framework developed to approximate its solution.

We begin by formulating the Fredholm integral equation with an oscillatory kernel. 
Let $I:=[-1,1]$, and denote by $C(I)$ the space of continuous, complex-valued functions on $I$, and by $C(I^2)$ the space of continuous complex-valued bivariate functions on $I^2$. Given a kernel function $K\in C(I^2)$ and a right-hand side $f\in C(I)$, we consider the oscillatory Fredholm integral equation 
\begin{equation}
y(s)-\int_{I} K(s, t) e^{i\kappa |s-t|} y(t)\mathrm{d}t=f(s), \ \  s\in I, \label{fredholm_equation}
\end{equation} 
where  $\kappa\geqslant 1$ is the wavenumber, and $y\in C(I)$ is the unknown solution. The kernel captures oscillatory interactions typical in wave propagation and scattering problems—such as acoustic, electromagnetic, or quantum waves. The exponential term $e^{i\kappa |s-t|}$ models a one-dimensional wave propagating with frequency $\kappa$ and decaying or reflecting based on the geometry of the domain.
We are particularly interested in the case where the wavenumber 
$\kappa\geq 100$. In this regime, equation \eqref{fredholm_equation} is classified as highly oscillatory, meaning the solution or kernel exhibits rapid oscillations over the domain. Such problems pose significant challenges for traditional numerical methods, which often require extremely fine discretizations or specialized techniques to maintain accuracy and stability.

% We define the integral operator $\mathcal{K}$ for $h\in C(I)$ as
% \begin{equation*}
% (\mathcal{K}h)(s):=\int_{I} K(s, t) e^{i\kappa |s-t|} h(t)\mathrm{d}t, \ \  s\in I. \label{def_K}
% \end{equation*} 
% Clearly,  $\mathcal{K}$ is a compact operator on $C(I)$. The integral equation \eqref{fredholm_equation} can then be expressed in its operator form as
% \begin{equation}
%     (\mathcal{I}-\mathcal{K})y=f, \label{fredholm_equation_operator}
% \end{equation}
% where $\mathcal{I}$ denotes the identity operator on $C(I)$. 

%To isolate the effects of the oscillatory kernel, we assume throughout this paper that the kernel function takes the simplified form
%\begin{equation*}
%    K(s,t):=\lambda,\ \ (s,t)\in I\times I
%\end{equation*}
%where $\lambda\in \mathbb{C}$ is a constant. 
%This assumption eliminates additional structural complexity, allowing us to focus on how effectively the method captures the oscillatory behavior arising from the exponential factor. All results in this paper extend without difficulty to equations with a general kernel $K$. 
The associated integral operator $\mathcal{K}$ is defined for $h \in C(I)$ by
\begin{equation*}
(\mathcal{K}h)(s):=\int_{I} K(s,t)e^{i\kappa |s-t|}h(t)\mathrm{d}t, \ \  s\in I. \label{def_K}
\end{equation*} 
Then, the Fredholm integral equation \eqref{fredholm_equation} can be expressed in operator form as
\begin{equation}
    (\mathcal{I}-\mathcal{K})y=f, \label{fredholm_equation_operator}
\end{equation}
where $\mathcal{I}$ denotes the identity operator on $C(I)$. Since $\mathcal{K}$ is compact, the operator equation \eqref{fredholm_equation_operator} admits a unique solution in $C(I)$ provided that $1$ is not an eigenvalue of $ \mathcal{K}$.
It is known \cite{wang2015oscillation} that solutions to equation \eqref{fredholm_equation_operator} generally exhibits oscillatory behavior, with the degree of oscillation proportional to the wavenumber.
The aim of this paper is to develop a numerical method that adaptively determines the number of layers in a DNN solution by leveraging the multi-grade deep learning framework to effectively handle such highly oscillatory problems.  

We now introduce DNNs used as approximate solutions of the integral equation.  We adopt the notation from \cite{Xu:2021aa,XuZhang2023}.
A DNN is a function formed by compositions of vector-valued functions, each defined by applying an activation function to an affine transformation.
Specifically, given a univariate activation function $\sigma:\mathbb{R}\to \mathbb{R}$, the corresponding vector-valued activation function ${\sigma}:\mathbb{R}^d\to \mathbb{R}^d$ is defined as
$$
{\sigma}(\mathbf{v}):= [\sigma(v_1), \sigma(v_2), \ldots, \sigma(v_d)]^T, \ \ \mbox{for}\ \ \mathbf{v} := [v_1, v_2, \ldots, v_{d}]^T\in \mathbb{R}^{d}.
$$
For vector-valued functions $f_j$, $j=1,2, \dots, n$, where the range of $f_j$ is contained within the domain of $f_{j+1}$, their  consecutive composition is denoted by
$$
\bigodot_{j=1}^n f_j:=f_n\circ f_{n-1}\circ \dots \circ f_1.
$$
%For a sequence $[m_j: j\in \mathbb{Z}_{n+1}]$ satisfying $m_0=1, m_n=2$. Among all $n+1$ layers, the first layer and the last layer are called the input layer and output layer respectively, while the other layers are called the hidden layer. 
Let $m_j$, $j=0,1, \dots, n$, be a sequence of positive integers with $m_0:=1, m_n:=2$. The first and last layers of the network are referred to as the input and output layers, respectively, while the intermediate layers are called hidden layers.
Given weight matrices $\mathbf{W}_j\in \mathbb{R}^{m_j\times m_{j-1}}$ and bias vectors $\mathbf{b}_j\in \mathbb{R}^{m_j}$, $j\in \mathbb{N}_{n}$, a DNN of depth $n$ is defined as
\begin{equation}
    \mathcal{N}_n(\theta;s) := \left(\mathbf{W}_n\bigodot_{j=1}^{n-1}\boldsymbol{\sigma}(\mathbf{W}_j\cdot+\mathbf{b}_j)+\mathbf{b}_n\right)(s), \quad s\in I, \label{dnn_output}
\end{equation}
where $\theta:=\{\mathbf{W}_j, \mathbf{b}_j\}_{j=1}^n$ represents the set of all trainable parameters. Furthermore, the output of last hidden layer is referred to as the feature of the DNN and is given by
\begin{equation}
    \mathcal{F}_{n-1}(\{\mathbf{W}_j, \mathbf{b}_j\}_{j=1}^{n-1};s) := \left(\bigodot_{j=1}^{n-1}\boldsymbol{\sigma}(\mathbf{W}_j\cdot+\mathbf{b}_j)\right)(s), \quad s\in I. \label{dnn_feature}
\end{equation}

The goal of this paper is to develop an adaptive deep learning method for approximating the solution of the integral equation \eqref{fredholm_equation} using a neural network representation of the form \eqref{dnn_output}. Since the solution is complex-valued, we introduce an operator $\mathcal{T}$ that maps a real-valued, two-dimensional vector-valued function to a complex-valued function. Specifically, for a function $\mathbf{f}(s) = [f_1(s), f_2(s)]^T$ with $f_1, f_2$ real-valued on $I$, we define
\[
(\mathcal{T} \mathbf{f})(s) := f_1(s) + i f_2(s), \quad s \in I.
\]
The operator $\mathcal{T}$ will be used to map real-valued DNN outputs to complex functions.
Using this transformation, we define the residual (loss) function associated with the integral equation \eqref{fredholm_equation_operator} as
\begin{flalign}
    e\left(\theta; s\right) := 
    (\mathcal{I} -  \mathcal{K})\, \mathcal{T} \mathcal{N}_n(\theta; \cdot)(s) - f(s), \quad s \in I, \label{hat_single_grade_error_function_K}
\end{flalign}
where $\theta := \{\mathbf{W}_j, \mathbf{b}_j\}_{j=1}^n$ denotes the network parameters. The optimal parameters $\theta^*$ are obtained as solutions of the optimization problem
\begin{flalign}
    \min_{\theta} \left\|e(\theta; \cdot)\right\|_2^2. \label{hat_DNN_opt_K}
\end{flalign}
This formulation represents the continuous analogue of the DNN training problem. The resulting minimization problem is generally non-convex, and the existence of a global minimizer is not guaranteed. Throughout this paper, we assume that all such minimization problems admit solutions.

Once $\theta^*$ is found from optimization problem \eqref{hat_DNN_opt_K}, the function 
$$
y^* := \mathcal{T} \mathcal{N}_n(\theta^*; \cdot)
$$ 
serves as a DNN approximation to the solution of equation \eqref{fredholm_equation}. However, due to the well-documented spectral bias of DNNs \cite{Rahaman2019}, directly optimizing \eqref{hat_DNN_opt_K} often fails to yield an accurate approximation, especially for highly oscillatory solutions.
To address this challenge, we propose an adaptive learning strategy based on the MGDL framework introduced in \cite{Xu:2023aa}. %This approach incrementally builds up the network structure to improve approximation quality while avoiding overfitting and inefficiencies associated with fixed-depth architectures.

We begin by describing the MGDL model.
Fix a positive integer $L\leq n$ and select $L$ positive integers $n_1$, $n_2$, $\dots$, $n_L$ such that $n-1=\sum_{l=1}^L n_l$. For grade 1, we use a neural network with $n_1$ layers. The error function at this grade is defined by
\begin{equation}
    e_1\left(\{\mathbf{W}_{j}, \mathbf{b}_{j}\}_{j=1}^{n_1}; s\right):=
    \left(f- (\mathcal{I}-\mathcal{K})\mathcal{T}\mathcal{N}_{n_1}(\{\mathbf{W}_{j}, \mathbf{b}_{j}\}_{j=1}^{n_1};\cdot)\right)(s),\quad s\in I.  \label{def_e1_con}
\end{equation} 
The optimal parameters $\{\mathbf{W}_{1,j}^*, \mathbf{b}_{1,j}^*\}_{j=1}^{n_1}$ are obtained by solving the optimization problem
%{\small
\begin{equation}
 \min \left\{\left\|e_1\left(\{\mathbf{W}_{j}, \mathbf{b}_{j}\}_{j=1}^{n_1}; \cdot \right)\right\|_2^2:  \mathbf{W}_{j}\in \mathbb{R}^{m_{1,j}\times m_{1,j-1}}, \mathbf{b}_{j}\in \mathbb{R}^{m_{1,j}}, j\in \mathbb{N}_{n_1}\right\},
 \label{multi_opt_problem_con1}
 \end{equation}%} 
with $m_{1,0}=1$ and $m_{1, n_1}=2$. Once the parameters $\{\mathbf{W}_{1,j}^*, \mathbf{b}_{1,j}^*\}_{j=1}^{n_1}$ are obtained,  we define the grade-1 feature function as
$$
\mathbf{g}_1(s):= \mathcal{F}_{n_1-1}\left(\{\mathbf{W}_{1,j}^*, \mathbf{b}_{1, j}^*\}_{j=1}^{n_1-1};s\right),\quad  s\in I,
$$
where $\mathcal{F}_{n_1-1}$ denotes the feature map up to the second-to-last layer.

The grade-1 solution component is given by
$$
\mathbf{f}_1(s):=\mathbf{W}_{1,2}^*\mathbf{g}_1(s)+\mathbf{b}_{1,2}^*, \quad s\in I,
$$
and the resulting optimal error function is
\begin{equation}
e_1^*(s) := e_1\left(\{\mathbf{W}^*_{1,j}, \mathbf{b}^*_{1,j}\}_{j=1}^{n_1}; s\right)\in C(I), \quad s\in I. \label{def_fin_e1_con}
\end{equation}

Suppose that the neural networks  $\mathbf{g}_l:I\to \mathbb{R}^{m_{l}}$, $\mathbf{f}_l:I\to \mathbb{R}^2$,  and the corresponding error function $e_l^*:I\to \mathbb{C}$ for some grade $l< L$ have been learned. We now proceed to construct grade $l+1$. We begin by defining the approximate solution accumulated up to grade $l$ as
\begin{equation}
    y^*_{l}:=\sum_{j=1}^{l} \mathcal{T} \mathbf{f}_j, \quad l\in \mathbb{N}_L. \label{def_y_l_con}
\end{equation}
The error function for grade $l+1$ is then given by
\begin{equation*}
e_{l+1}\left(\left\{\mathbf{W}_{j}, \mathbf{b}_{ j}\right\}_{j=1}^{n_{l+1}}; s\right):= \left(f-(\mathcal{I}-\mathcal{K})(y^*_{l}+\mathcal{T}\mathcal{N}_{2}\left(\{\mathbf{W}_{j}, \mathbf{b}_{j}\}_{j=1}^{n_{l+1}};\mathbf{g}_{l}(\cdot)\right)\right)(s),\label{def_e_q+1_con}
\end{equation*}
for $s\in I$.
To learn the parameters for grade $l+1$,
we solve the following optimization problem to obtain the minimizer $\{\mathbf{W}_{l+1,j}^*, \mathbf{b}_{l+1,j}^*\}_{j=1}^{n_{l+1}}$:
\begin{flalign}  
    &\min \left\{\left\|e_{l+1}\left(\left\{\mathbf{W}_{j}, \mathbf{b}_{j}\right\}_{j=1}^{n_{l+1}}; \cdot\right)\right\|_2^2: \right. \nonumber \\  
    &\quad \quad \left.\mathbf{W}_{j}\in \mathbb{R}^{m_{l+1,j}\times m_{l+1,j-1}}, \mathbf{b}_{j}\in \mathbb{R}^{m_{l+1,j}}, j\in \mathbb{N}_{n_{l+1}}\right\}, \label{mult_weight_con}  
\end{flalign}
with $m_{l+1,0}=m_{l,n_l-1}$ and $m_{l+1, n_{l+1}}=2$. Once the optimal parameters are learned, we define the feature representation of grade $l+1$ as
\begin{equation}
    \mathbf{g}_{l+1}(s):= \mathcal{F}_{n_{l+1}-1}\left(\left\{\mathbf{W}_{l+1,j}^*, \mathbf{b}_{l+1, j}^*\right\}_{j=1}^{n_{l+1}-1};\mathbf{g}_l(s)\right),\quad s\in I, \label{feature_grade_l_1_con}
\end{equation} 
the corresponding solution component as
$$ 
\mathbf{f}_{l+1}(s):=\mathbf{W}_{l+1,2}^*\mathbf{g}_{l+1}(s)+\mathbf{b}_{l+1. 2}^*,\quad s\in I,
$$
and the resulting optimal error as 
\begin{equation}
e^*_{l+1}(s) := e_{l+1}\left(\left\{\mathbf{W}^*_{l+1,j}, \mathbf{b}^*_{l+1,j}\right\}_{j=1}^{n_{l+1}}; s\right)\in C(I), \quad s\in I. \label{def_fin_e_q+1_con}
\end{equation}
This process is repeated for each grade  $l<L$. The learning procedure may terminate early at some grade $l<L$ if the error norm $\|e_l^*\|_2$ falls below a prescribed tolerance. A theoretical justification for this stopping criterion will be presented in the next section. 

Finally, the multi-grade DNN approximation to the solution is defined by
$$
y^*_{L}:=\sum_{l=1}^{L} \mathcal{T} \mathbf{f}_l.
$$ 
In the formulation above, we assume that the errors $\left\|e_{l}\left(\left\{\mathbf{W}_{j}, \mathbf{b}_{j}\right\}_{j=1}^{n_l}; \cdot\right)\right\|_2^2$ are computed exactly, referring to this framework as the {\it continuous multi-grade deep learning model}.
In the next section, we establish both lower and upper bounds for the error $\|y-y^*_{l}\|_2$ in terms of the training error $e_l^*$ within the continuous MGDL framework.

The continuous MGDL model described above provides a theoretical foundation but is not directly implementable in practice. In real-world applications, evaluating the objective function and training the model require numerical integration, which introduces discretization. This leads naturally to a discrete version of the MGDL model. In what follows, we briefly recall the discrete MGDL framework developed in our previous work \cite{jiang2024deep}, which serves as the basis for practical implementation.

%%%%%%%%%%%%%%%%
The continuous MGDL requires solving the optimization problems \eqref{multi_opt_problem_con1} and \eqref{mult_weight_con}. Implementing it requires computing
\begin{itemize}
    \item the $L_2$-norm of functions involved
    \item the integral operator $\mathcal{K}$.
\end{itemize}
To this end, we assume that distinct points $\{x_j\}_{j=1}^N$ in $I$ are chosen.
The $L_2$-norm $\|g\|_2$ of a function $g \in C(I)$ is approximated by the discrete semi-norm
\begin{equation*}
\|g\|_N := \sqrt{\frac{1}{N} \sum_{j=1}^N |g(x_j)|^2}, \quad g \in C(I),
\end{equation*}
where $|\cdot|$ denotes the modulus of a complex number. This semi-norm is derived from the $\ell_2$-norm on $\mathbb{C}^N$ but is not a true norm on $C(I)$, since $\|g\|_N = 0$ does not necessarily imply $g = 0$.

For any $g \in C(I)$, define the vector $\mathbf{v}_g := [g(x_j)]_{j=1}^N \in \mathbb{C}^N$. Then the relationship between the discrete semi-norm and the standard $\ell_2$-norm is given by
\begin{equation}
\|g\|_N = \frac{\|\mathbf{v}_g\|_{\ell_2}}{\sqrt{N}}. \label{relation_N_2}
\end{equation}
%%%%%%%%%%%%%%%%%

We assume that a numerical integration scheme has been chosen with quadrature nodes $\{s_j\}_{j=0}^{p_\kappa}$ to construct a discrete operator $\mathcal{K}_{p_\kappa}$ that approximates the continuous integral operator $\mathcal{K}$, where $p_\kappa$ is a positive integer depending on $\kappa$. Replacing $\mathcal{K}$ in equation \eqref{fredholm_equation_operator} with $\mathcal{K}_{p_\kappa}$ and collocating the resulting equation at nodes $\{x_j\}_{j=1}^N$ yields the discrete system:
\begin{equation}\label{discrete_operator_Equation}
((\mathcal{I} - \mathcal{K}_{p_\kappa})y_h)(x_j) = f(x_j), \quad j=1,2,\dots,N,
\end{equation}
where $y_h$ is an approximation to the true solution $y$. In general, the quadrature and collocation nodes need not coincide. However, by the theory of collectively compact operators \cite{anselone1971collectively}, if $\mathcal{K}_{p_\kappa}$ converges pointwise to $\mathcal{K}$ and $p_\kappa$ is sufficiently large, then the discrete system \eqref{discrete_operator_Equation} admits a unique solution for any given right-hand side $f$, provided the quadrature and collocation nodes coincide.

%%%%%% recall the single-grade learning model
With the discrete $L_2$-norm and the discrete integral operator $\mathcal{K}_{p_{_\kappa}}$, we recall the single-grade learning model introduced in  \cite{jiang2024deep}. Motivated by the discrete system \eqref{discrete_operator_Equation}, we define the corresponding loss function as
 \begin{flalign*}
    \tilde{e}\left(\{\mathbf{W}_j, \mathbf{b}_j\}_{j=1}^{n};s\right):= \left(f-
    (\mathcal{I}-\mathcal{K}_{p_{_\kappa}}) \mathcal{T}\mathcal{N}_n(\{\mathbf{W}_j, \mathbf{b}_j\}_{j=1}^{n};\cdot)\right) (s),\quad  s\in I, 
\end{flalign*}
which serves as a discrete approximation to the contonuous loss function defined in equation \eqref{hat_single_grade_error_function_K}.
The optimal parameters $\{\tilde{\mathbf{W}}^*_j, \tilde{\mathbf{b}}^*_j\}_{j=1}^n$ are obtained by solving the minimization problem
\begin{flalign}
\min_{\{\mathbf{W}_j, \mathbf{b}_j\}_{j=1}^n} \left\|\tilde{e}\left(\{\mathbf{W}_j, \mathbf{b}_j\}_{j=1}^{n};\cdot\right)\right\|_{N}^2.  \label{sgl_opt}
\end{flalign}
Note that employing the discrete norm $\|\cdot\|_N$ in \eqref{sgl_opt} restricts the optimization problem to the discrete points used to define this norm. Therefore, the optimization problem \eqref{sgl_opt} is discrete.
The corresponding  numerical solution is given by 
\begin{equation*}
     \tilde{y}^*:=\mathcal{T}\mathcal{N}_{n}(\{\tilde{\mathbf{W}}_j^*, \tilde{\mathbf{b}}_j^*\}_{j=1}^n;\cdot) \label{def_fin_y_star}
\end{equation*}
and the associated error function is defined as
\begin{equation*}
\tilde{e}^*(s):= \tilde{e}\left(\{\tilde{\mathbf{W}}_j^*, \tilde{\mathbf{b}}_j^*\}_{j=1}^{n};s\right),\quad s\in I. 
\end{equation*}

%%%%%%%%%%%%%%%%% recall the discrete multi-grade learning model

We now introduce the discrete MGDL model for numerically solving equation \eqref{fredholm_equation}. As in the continuous MGDL model, we are given a sequence of positive integers $n_l$, for $l=1,2, \dots, L$, such that 
$n=\sum_{l=1}^L n_l$.  Based on this setup, the model solves $L$ interconnected minimization problems. 

For grade 1, define the error function by
\begin{equation*}
    \tilde{e}_1\left(\{\mathbf{W}_{j}, \mathbf{b}_{j}\}_{j=1}^{n_1}; s\right):=
    \left(f- (\mathcal{I}-\mathcal{K}_{p_{_\kappa}})\mathcal{T}\mathcal{N}_{n_1}(\{\mathbf{W}_{j}, \mathbf{b}_{j}\}_{j=1}^{n_1};\cdot)\right)(s)\in C(I), s\in I.  \label{def_e1_dis}
\end{equation*} 
We then solve the following optimization problem:
\begin{equation}
 \min \left\{\left\|\tilde{e}_1\left(\{\mathbf{W}_{j}, \mathbf{b}_{j}\}_{j=1}^{n_1}; \cdot \right)\right\|_{N}^2:  \mathbf{W}_{j}\in \mathbb{R}^{m_{1,j}\times m_{1,j-1}}, \mathbf{b}_{j}\in \mathbb{R}^{m_{1,j}}, j\in \mathbb{N}_{n_1}\right\}, \label{multi_opt_problem_dis_grade_1}
 \end{equation}
with $m_{1,0}=1$ and $m_{1, n_1}=2$, yielding the optimal parameters  $\{\tilde{\mathbf{W}}_{1,j}^*, \tilde{\mathbf{b}}_{1,j}^*\}_{j=1}^{n_1}$. 
Define the feature map and solution component of grade 1 as
$$
\tilde{\mathbf{g}}_1(s):= \mathcal{F}_{n_1-1}\left(\{\tilde{\mathbf{W}}_{1,j}^*, \tilde{\mathbf{b}}_{1, j}^*\}_{j=1}^{n_1-1};s\right),\quad  s\in I, 
$$
and
$$
\tilde{\mathbf{f}}_1(s):=\tilde{\mathbf{W}}_{1,n_1}^*\mathbf{g}_1(s)+\tilde{\mathbf{b}}_{1,n_1}^*, \quad s\in I,
$$
respectively, where $\mathcal{F}_{n_1-1}$ is defined as in equation \eqref{dnn_feature}. The corresponding error is given by 
\begin{equation}
\tilde{e}_1^*(s) := \tilde{e}_1\left(\{\tilde{\mathbf{W}}^*_{1,j}, \tilde{\mathbf{b}}^*_{1,j}\}_{j=1}^{n_1}; s\right)\in C(I), \quad s\in I. \label{def_fin_e1_dis}
\end{equation}

Suppose that the neural networks  $\tilde{\mathbf{g}}_l:I\to \mathbb{R}^{m_{l,n_l-1}}$, $\tilde{\mathbf{f}}_l:I\to \mathbb{R}^2$, and $\tilde{e}_l^*:I\to \mathbb{C}$ have been computed for grade $l<L$. To compute the next grade, define the error function for grade $l+1$ as: for $s\in I$,
\begin{equation}
\tilde{e}_{l+1}\left(\left\{\mathbf{W}_{j}, \mathbf{b}_{ j}\right\}_{j=1}^{n_{l+1}}; s\right):= \tilde{e}^*_{l}(s)-\left[(\mathcal{I}-\mathcal{K}_{p_{_\kappa}})\mathcal{T}\mathcal{N}_{n_{l+1}}\left(\{\mathbf{W}_{j}, \mathbf{b}_{j}\}_{j=1}^{n_{l+1}};\mathbf{g}_{l}(\cdot)\right)\right](s).  \label{def_e_q+1_dis}
\end{equation}
We then consider the optimization problem  
\begin{equation}\label{mult_weight_dis}
    \min_{\substack{\mathbf{W}_{j}\in \mathbb{R}^{m_{l+1,j}\times m_{l+1,j-1}} \\[2pt] \mathbf{b}_{j}\in \mathbb{R}^{m_{l+1,j}}, \; j\in \mathbb{N}_{n_{l+1}}}}
    \left\| \tilde{e}_{\,l+1}\!\left(\{\mathbf{W}_{j}, \mathbf{b}_{j}\}_{j=1}^{n_{l+1}}; \cdot\right)\right\|_{N}^{2},
\end{equation}
where $m_{l+1,0}=m_{l,n_l-1}$ and $m_{l+1,n_{l+1}}=2$.  
The solution of \eqref{mult_weight_dis} yields the optimal parameters  
\[
\{\tilde{\mathbf{W}}_{l+1,j}^*, \tilde{\mathbf{b}}_{l+1,j}^*\}_{j=1}^{n_{l+1}}.
\]
Then define the feature and the solution component of grade $l+1$ by
\begin{equation*}
    \tilde{\mathbf{g}}_{l+1}(s):= \mathcal{F}_{n_{l+1}-1}\left(\left\{\tilde{\mathbf{W}}_{l+1,j}^*, \tilde{\mathbf{b}}_{l+1, j}^*\right\}_{j=1}^{n_{l+1}-1};\tilde{\mathbf{g}}_l(s)\right),\quad s\in I,
\end{equation*}
and 
$$ \tilde{\mathbf{f}}_{l+1}(s):=\tilde{\mathbf{W}}_{l+1,n_{l+1}}^*\tilde{\mathbf{g}}_{l+1}(s)+\tilde{\mathbf{b}}_{l+1. n_{l+1}}^*,\quad s\in I.
$$
The corresponding optimal error is defined by  
\begin{equation}
\tilde{e}^*_{l+1}(s) := \tilde{e}_{l+1}\left(\left\{\tilde{\mathbf{W}}^*_{l+1,j}, \tilde{\mathbf{b}}^*_{l+1,j}\right\}_{j=1}^{n_{l+1}}; s\right)\in C(I), \quad s\in I. \label{def_fin_e_q+1_dis}
\end{equation}
After completing all $L$ grades, the final  discrete multi-grade DNN approximation to $y$ is given by
$$
\tilde{y}^*_{L}:=\sum_{l=1}^{L} \mathcal{T} \tilde{\mathbf{f}}_l\in C(I).
$$

%As for the analysis of continuous multi-grade learning model,
%it has been proved the norm of optimal error is a non-increasing function about the grades for function approximation problem, see Theorem 4.4 in paper \cite{Xu:2023aa}. In our task, a similar result will be given in next section. Readers should note that, in describing this algorithm, due to the nature of gradient-based optimization algorithms, we only assume that the optimization problems \cref{multi_opt_problem_con1}, \cref{mult_weight_con} being solved yields a local minimum. Therefore, the proofs in the next chapter are conducted under conditions weaker than those presented in the paper \cite{Xu:2023aa}.

\section{Error Analysis of the Continuous MGDL Model}\label{Analysis-of-Cont}

In this section, we derive upper and lower bounds for the error of the continuous MGDL model in terms of the training loss. We begin by proving that the norm of the optimal error is non-increasing with respect to the number of grades. Leveraging this monotonicity, we then establish that, for a continuous MGDL model with $L$ grades, the approximation error at grade $l+1$ is bounded by a constant multiple of the error at grade $l$, for all $l=1,2, \dots, {L-1}$.

We note that the final layer of a neural network performs an affine transformation without an activation function. Since affine transformations are well understood analytically, we isolate this final transformation to facilitate the theoretical analysis.

For each grade $l=1,2, \dots, L$, define the following parameter spaces
$$\Theta_{l, 1}:=\left\{\left(\{\mathbf{W}_{j}, \mathbf{b}_{j}\}_{j=1}^{n_{l}-1}\right):\mathbf{W}_{j}\in \mathbb{R}^{m_{l,j}\times m_{l,j-1}}, \mathbf{b}_{j}\in \mathbb{R}^{m_{l,j}}, j\in \mathbb{N}_{n_{l}-1}\right\}$$
%equipped with  norm
%$$ \left\|\left(\{\mathbf{W}_{j}, \mathbf{b}_{j}\}_{j=1}^{n_{l}-1}\right)\right\|_{\Theta_{l,1}}:=\left(\sum_{j=1}^{n_{l}-1}\left(\|\mathbf{W}_{j}\|_2^2+\|\mathbf{b}_{j}\|_2^2\right)\right)^{\frac{1}{2}}, \  \mbox{for all}\   \left(\{\mathbf{W}_{j}, \mathbf{b}_{j}\}_{j=1}^{n_{l}-1}\right)\in \Theta_{l,1}$$
and
$$\Theta_{l,2}:=\{(\mathbf{W}, \mathbf{b}):\mathbf{W}\in \mathbb{R}^{m_{l,n_{l}}\times m_{l,n_l-1}}, \mathbf{b}\in \mathbb{R}^{m_{l,n_{l}}}\}.$$
%equipped with norm
%$$ \left\|\left(\mathbf{W}, \mathbf{b}\right)\right\|_{\Theta_{l,2}}:=\left(\|\mathbf{W}\|_2^2+\|\mathbf{b}\|_2^2\right)^{\frac{1}{2}}, \quad \textrm{for \,all\,}\left(\mathbf{W}, \mathbf{b}\right)\in \Theta_{l,2}.$$
As the feature function $\mathbf{g}_l$ of grade $l$ is fixed, an operator $\mathcal{M}_{l}: \Theta_{l,2}\to C(I)$ is defined as
\begin{equation} 
\mathcal{M}_{l}((\mathbf{W}, \mathbf{b})):= (\mathcal{I}-\mathcal{K})\mathcal{T}(\mathbf{W}\mathbf{g}_{l}(\cdot)+\mathbf{b})\in C(I), \quad  (\mathbf{W},\mathbf{b})\in \Theta_{l,2}.  \label{def_M_con}
\end{equation}
We let 
$$
\mathcal{M}_{l}(\Theta_{l, 2}):=\{\mathcal{M}_{l}((\mathbf{W}, \mathbf{b})): (\mathbf{W}, \mathbf{b})\in \Theta_{l,2}\}
$$
and observe that $\mathcal{M}_{l}(\Theta_{l, 2})$ is a linear space.

Meanwhile, for any $l\in \mathbb{N}_L$, define
\begin{equation*}
   z_l:=\mathcal{M}_{l}\left(\left(\mathbf{W}^*_{l,n_{l}}, \mathbf{b}^*_{l,n_{l}}\right)\right)\in C(I). \label{def_z_l}
\end{equation*}
Clearly, 
$
z_l\in \mathcal{M}_{l}(\Theta_{l, 2}).
$
In the  following analysis, we will prove that for any $l=1,2, \dots, {L-1}$, the function $z_{l+1}$ is the unique best approximation to the error $e_l^*$ from the linear space $\mathcal{M}_{l+1}(\Theta_{l+1,2})$. To establish this result, we first present two auxiliary lemmas. For generality and to ensure their applicability to both the continuous and discrete MGDL models, we provide proofs in a unified framework.

The following lemma can be used to fix the parameters in a neural network except for the last layer.

\begin{lemma} \label{lem_mult1}
Let $\mathcal{B}_1$ and $\mathcal{B}_2$ be two given sets, and let $F : \mathcal{B}_1 \times \mathcal{B}_2 \to [0, +\infty)$ be a target function. If $(\theta_1^*, \theta_2^*) \in \mathcal{B}_1 \times \mathcal{B}_2$ is a  local minimizer of the problem
\begin{equation}
    \min_{(\theta_1, \theta_2) \in \mathcal{B}_1 \times \mathcal{B}_2} F(\theta_1, \theta_2), \label{opt1}
\end{equation}
then $\theta_2^* \in \mathcal{B}_2$ is a local minimizer of the reduced problem
\begin{equation}
    \min_{\theta_2 \in \mathcal{B}_2} F(\theta_1^*, \theta_2). \label{opt2}
\end{equation}
\end{lemma}

\begin{proof}
We prove by contradiction. Suppose $\theta_2^*$ is not a local minimizer of \eqref{opt2}. Then there exists a sequence $\{s_j\} \subset \mathcal{B}_2$ such that $s_j \to \theta_2^*$ and
\begin{equation}
F(\theta_1^*, s_j) < F(\theta_1^*, \theta_2^*) \quad \text{for all } j \in \mathbb{N}. \label{lem_mult1_eq}
\end{equation}
Define $t_j := (\theta_1^*, s_j) \in \mathcal{B}_1 \times \mathcal{B}_2$. Then $t_j \to (\theta_1^*, \theta_2^*)$ and, by \eqref{lem_mult1_eq},
\[
F(t_j) < F(\theta_1^*, \theta_2^*) \quad \text{for all } j \in \mathbb{N},
\]
which contradicts the assumption that $(\theta_1^*, \theta_2^*)$ is a local minimizer of \eqref{opt1}. Therefore, $\theta_2^*$ must be a local minimizer of \eqref{opt2}.
\end{proof}

The following lemma shows that, owing to the convexity of the objective function, any local minimizer in the parameter space $\mathcal{B}$ for the semi-norm error corresponds to the global best approximation in the function subspace $\mathcal{U}(\mathcal{B})$. This result is formulated in a general semi-normed setting and will later be applied to the discrete MGDL model.

\begin{lemma} \label{lem_multi2}
Let $\mathcal{X}$ be a linear space equipped with a semi-norm $|\cdot|$, and let $\mathcal{U} : \mathcal{B} \to \mathcal{X}$ be a linear operator from a Banach space $\mathcal{B}$ into $\mathcal{X}$. Suppose that $s \in \mathcal{X}$. If $t^* \in \mathcal{B}$ is a local minimizer of the optimization problem
\begin{equation}
    \min_{t \in \mathcal{B}} \, |s - \mathcal{U}t|, \label{lem_multi2_opt}
\end{equation}
then $\mathcal{U}t^* \in \mathcal{U}(\mathcal{B})$ is the best approximation to $s$ from the linear subspace $\mathcal{U}(\mathcal{B})$ with respect to the semi-norm $|\cdot|$.
\end{lemma}
\begin{proof}
The function $t \mapsto |s - \mathcal{U}t|$ is convex on $\mathcal{B}$ because it is the composition of a linear map and a semi-norm. Hence, any local minimizer is also a global minimizer.

Since $t^* \in \mathcal{B}$ is a local minimizer of \eqref{lem_multi2_opt}, it must also be a global minimizer. Thus,
\[
    |s - \mathcal{U}t^*| \leq |s - \mathcal{U}t|, \quad \text{for all } t \in \mathcal{B}.
\]
This means that $\mathcal{U}t^* \in \mathcal{U}(\mathcal{B})$ minimizes the semi-norm distance to $s$ over the set $\mathcal{U}(\mathcal{B})$, i.e., it is the best approximation to $s$ from the linear subspace $\mathcal{U}(\mathcal{B})$ with respect to $|\cdot|$.
\end{proof}

The following result for our model follows directly from the preceding two lemmas.

\begin{lemma} \label{lem_multi3_feature_space}
For any $l\in \mathbb{N}_{L-1}$, $z_{l+1}$ is the unique best approximation from the linear space $\mathcal{M}_{l+1}(\Theta_{l+1,2})$ to $e_l^*$.
\end{lemma}

\begin{proof}
For any $l \in \mathbb{N}_{L-1}$, we apply Lemma~\ref{lem_mult1} with the following choices:
$\mathcal{B}_1 := \Theta_{l+1,1}$, $\mathcal{B}_2 := \Theta_{l+1,2}$, 
\[\theta_1^* := \left(\{\mathbf{W}_{l+1,j}^*, \mathbf{b}_{l+1,j}^*\}_{j=1}^{n_{l+1}-1}\right), \quad 
\theta_2^* := \left(\mathbf{W}_{l+1,n_{l+1}}^*, \mathbf{b}_{l+1,n_{l+1}}^*\right).
\]
Let $F: \Theta_{l+1,1} \times \Theta_{l+1,2} \to [0, +\infty)$ be the objective function defined by the training error:
\[
F(\theta_1, \theta_2) := \left\|e_{l+1}(\theta_1 \times \theta_2; \cdot)\right\|_2.
\]
Since $(\theta_1^*, \theta_2^*)$ is a local minimizer of $F$, Lemma~\ref{lem_mult1} implies that $\theta_2^*$ is a local minimizer of the reduced problem
\[
\min_{\theta \in \Theta_{l+1,2}} \left\|e_{l+1}(\theta_1^* \times \theta; \cdot)\right\|_2.
\]
Using the definition of $e_{l+1}$ from \eqref{mult_weight_con}, this is equivalent to minimizing
\[
\min_{\theta \in \Theta_{l+1,2}} \left\|e_l^* - (\mathcal{I} -  \mathcal{K})\mathcal{T} \mathcal{N}_{n_{l+1}}(\theta_1^* \times \theta; \cdot) \right\|_2.
\]
By the definition of the operator $\mathcal{M}_{l+1}$ in \eqref{def_M_con}, this reduces to the optimization problem
\[
\min_{\theta \in \Theta_{l+1,2}} \left\|e_l^* - \mathcal{M}_{l+1}(\theta)\right\|_2.
\]

Now, apply Lemma~\ref{lem_multi2} with the settings:
\[
\mathcal{X} := C(I), \quad |\cdot| := \|\cdot\|_2, \quad \mathcal{U} := \mathcal{M}_{l+1}, \quad s := e_l^*, \quad t^* := \theta_2^*.
\]
Then, $\mathcal{M}_{l+1}(\theta_2^*) = z_{l+1}$ is the best approximation to $e_l^*$ in %in the subspace 
$\mathcal{M}_{l+1}(\Theta_{l+1,2}) \subset C(I)$ with respect to the $L_2$ norm.

Finally, uniqueness follows from the strict convexity of the $L_2$ norm and the fact that $\mathcal{M}_{l+1}(\Theta_{l+1,2})$ is a closed linear subspace of the Hilbert space $L_2(I)$.
\end{proof}

We now show that the sequence $\|e_l^*\|_2$, $l=1,2, \dots, L$, is non-increasing. This result extends Theorem 3 of \cite{Xu:2023aa}, which was originally established for function approximation, to the numerical solution of integral equations. In addition, Theorem 3 in \cite{Xu:2023aa} considered an MGDL architecture in which the output layer is retained when a new grade is added. In contrast, the result presented here applies to an MGDL framework where the output layer from the previous grade is removed when a new grade is added, a setting that more closely reflects practical implementations.

\begin{proposition} \label{thm_multi_con_e_j}
For each $l=1,2, \dots, {L-1}$, either $z_{l+1}=0$ or $\|e^*_{l+1}\|_2< \|e^*_{l}\|_2$.
\end{proposition}
\begin{proof}
For each  $l=1,2, \dots, {L-1}$, suppose that $z_{l+1}\neq 0$. By the definition  of $e_{l+1}^*$ in \cref{def_fin_e_q+1_con}, we have 
$$
e_{l+1}^*=e_l^*-z_{l+1}.
$$ 
Since $0$ and $z_{l+1}$ are distinct elements of $\mathcal{M}_{l+1}(\Theta_{l+1, 2})$, and $z_{l+1}$ is the unique best approximation of $e_l^*$ from $\mathcal{M}_{l+1}(\Theta_{l+1,2})$ by Lemma \ref{lem_multi3_feature_space},
it follows that
$$
\|e^*_{l+1}\|_2=\|e^*_{l}-z_{l+1}\|_2< \|e^*_{l}-0\|_2=\|e^*_l\|_2.
$$
This completes the proof.
\end{proof}

If $1$ is not the eigenvalue of the linear operator $\mathcal{K}$, then by the Fredholm Alternative Theorem, the operator $\mathcal{I}-\mathcal{K}$ is invertible, and its inverse $(\mathcal{I}- \mathcal{K})^{-1}$ is bounded. This fact enables the derivation of a theorem that characterizes the error behavior across successive grades. Before stating the theorem, we first present the following lemma, which describes the relationship between $y_l^*$ and $e_l^*$ for $l=1,2, \dots, L$.

\begin{lemma} \label{lem_e_l_star}
For each $l=1,2, \dots, L$, it holds that 
\[
e_l^* = (\mathcal{I} - \mathcal{K})(y - y_l^*).
\]
\end{lemma}
\begin{proof}
We prove this result by induction on $l$. From the definition of $e_1^*$ in \cref{def_fin_e1_con} and using $f = (\mathcal{I} - \mathcal{K})y$, we have
\begin{equation}
e_1^* = f - (\mathcal{I} - \mathcal{K})\mathcal{T}\mathbf{f}_1 = (\mathcal{I} - \mathcal{K})(y - \mathcal{T}\mathbf{f}_1). \label{e_1_star}
\end{equation}
Now, assume that for some $l=1,2, \dots, {L-1}$, the identity 
\[
e_l^* = (\mathcal{I} - \mathcal{K})\left(y - \sum_{j=1}^l \mathcal{T}\mathbf{f}_j\right)
\]
holds. Then, by the recursive definition of $e_{l+1}^*$ in \cref{def_fin_e_q+1_con}, we have
\[
e_{l+1}^* = e_l^* - (\mathcal{I} - \mathcal{K})\mathcal{T}\mathbf{f}_{l+1}.
\]
Substituting the induction hypothesis into the right-hand side yields
\[
e_{l+1}^* = (\mathcal{I} - \mathcal{K})\left(y - \sum_{j=1}^l \mathcal{T}\mathbf{f}_j\right) - (\mathcal{I} - \mathcal{K})\mathcal{T}\mathbf{f}_{l+1} = (\mathcal{I} - \mathcal{K})\left(y - \sum_{j=1}^{l+1} \mathcal{T}\mathbf{f}_j\right).
\]
Thus, by induction, the formula holds for all $l=1,2, \dots, L$. Using the definition of $y_l^*$ from \cref{def_y_l_con}, the result follows.
\end{proof}

We now state the theorem that quantifies the approximation error in the continuous multi-grade learning model.

\begin{theorem}
    \label{the:bound_result1_con}
Suppose that $1$ is not an eigenvalue of the linear operator $\mathcal{K}$. Then the operator $\mathcal{I} - \mathcal{K}$ is invertible, and for each $l=1,2, \dots, {L-1}$, the following error estimate holds: 
$$
\|y-y_{l+1}^*\|_2\leqslant c\|y-y_{l}^*\|_2,  
$$
where the constant $c:=\|(\mathcal{I}- \mathcal{K})^{-1}\| \|\mathcal{I}- \mathcal{K}\|$.
\end{theorem}
\begin{proof}
Since $1$ is not an eigenvalue of $\mathcal{K}$, the operator $\mathcal{I}- \mathcal{K}$ is invertible, and its inverse $(\mathcal{I}- \mathcal{K})^{-1}$ is bounded. Define constants $$
c_1:=\|(\mathcal{I}- \mathcal{K})^{-1}\|^{-1}>0\ \mbox{and}\  c_2:=\|\mathcal{I}- \mathcal{K}\|>0.
$$
Then, for any $h\in C(I)$, the following inequality holds:
\begin{equation*} 
   c_1\|h\|_2 \leqslant\|(\mathcal{I}- \mathcal{K})h\|_2 \leqslant c_2\|h\|_2.  
\end{equation*}
Applying this inequality to $h:=y-y_{l+1}^*$ and $h:=y-y_{l}^*$ yields
\begin{equation}
       c_1 \|y-y_{l+1}^*\|_2 \leqslant\|(\mathcal{I}- \mathcal{K})(y-y_{l+1}^*)\|_2 \label{the:bound_result1_con_equ1}
\end{equation} 
and
\begin{equation}
       \|(\mathcal{I}- \mathcal{K})(y-y_{l}^*)\|_2 \leqslant c_2\|y-y_{l}^*\|_2.\label{the:bound_result1_con_equ2}
\end{equation}
By Lemma \ref{lem_e_l_star}, we can rewrite \cref{the:bound_result1_con_equ1} and \cref{the:bound_result1_con_equ2} as
\begin{equation}
    c_1\|y-y_{l+1}^*\|_2\leqslant \|e_{l+1}^*\|_2, \quad \|e_l^*\|_2\leqslant c_2 \|y-y_{l}^*\|_2. \label{the:bound_result1_con_equ}
\end{equation}
From  Proposition \ref{thm_multi_con_e_j}, we know that either $e_{l+1}^*=e_l^*$ (which occurs only if  $z_{l+1}=0$), or $\|e_{l+1}^*\|_2< \|e_{l}^*\|_2$. In either case, we have
$$\|e_{l+1}^*\|_2\leqslant \|e_{l}^*\|_2.$$ 
Combining this with \cref{the:bound_result1_con_equ}, we obtain
$$
 c_1\|y-y_{l+1}^*\|_2\leqslant \|e_{l+1}^*\|_2, \leq\|e_l^*\|_2\leqslant c_2 \|y-y_{l}^*\|_2,  
$$
which implies 
$$
\|y-y_{l+1}^*\|_2\leqslant  \frac{c_2}{c_1} \|y-y_{l}^*\|_2.
$$
and setting 
$$
c:=c_2/c_1=\|(\mathcal{I}- \mathcal{K})^{-1}\|\|\mathcal{I}- \mathcal{K}\|,
$$
we obtain the desired result.
\end{proof}

Theorem~\ref{the:bound_result1_con} provides a worst-case stability bound showing that the solution error at each additional grade of the MGDL network is controlled by a constant factor. In particular, it guarantees that the solution error cannot grow unbounded as new grades are added, thereby ensuring the stability of the multi-grade refinement process. The constant
\[
c := \|(\mathcal{I} - \mathcal{K})^{-1}\|\,\|\mathcal{I} - \mathcal{K}\|
\]
is the condition number of the operator \(\mathcal{I} - \mathcal{K}\) and satisfies \(c \ge 1\). While the bound
\[
\|y - y_{l+1}^*\|_2 \le c \|y - y_l^*\|_2
\]
does not by itself guarantee a strict monotonic decrease of the solution error, it ensures that adding grades consistently maintains stability and, for moderately conditioned operators, allows error reduction to be effectively achieved through the multi-grade refinement.

Our goal is to approximate the true solution \( y \), but the actual error \( \|y - y_l^*\|_2 \) at grade \( l \in \mathbb{N} \) is generally not accessible during the training process, as it depends on the unknown ground truth solution \( y \). In contrast, the training error \( \|e_l^*\|_2 \), which quantifies the discrepancy between the data \( f \) and the network output after applying the operator \( (\mathcal{I} - \mathcal{K}) \), is fully computable from the model. Understanding how this observable training error relates to the true (but hidden) solution error is crucial for evaluating model performance and guiding the training process. This relationship is rigorously characterized in the following theorem.

%During training, the true solution error $\|y-y_l^*\|_2$ for $l\in \mathbb{N}$ is not directly observable. However, we can compute the training error $\|e^*_l\|_2$ explicitly. The relationship between these two quantities is established in the following theorem.

\begin{theorem}
    \label{thm:bound_result2_con}
If $1$ is not an eigenvalue of the linear operator $\mathcal{K}$, then for each $l=1,2, \dots, L$, there holds 
$$
\|\mathcal{I}-\mathcal{K}\|^{-1}\|e_l^*\|_2 \leqslant\|y-y_l^*\|_2 \leqslant \|(\mathcal{I}- \mathcal{K})^{-1}\|\|e_l^*\|_2.
$$
\end{theorem}
\begin{proof}
Since $1$ is not an eigenvalue of $\mathcal{K}$, the operator $\mathcal{I}- \mathcal{K}$ is invertible, and its inverse $(\mathcal{I}- \mathcal{K})^{-1}$ is bounded. Thus, for each $l\in \mathbb{N}_L$,  we obtain the inequality
\begin{equation*} 
\|(\mathcal{I}-\mathcal{K})^{-1}\|^{-1}\|y-y_l^*\|_2 \leqslant\|(\mathcal{I}- \mathcal{K})(y-y_l^*)\|_2 \leqslant \|\mathcal{I}- \mathcal{K}\|\|y-y_l^*\|_2. \end{equation*}
Applying Lemma \ref{lem_e_l_star}, which states that  $e_l^*= (\mathcal{I}- \mathcal{K})(y-y_{l}^*)$, we rewrite the above inequality as
$$
\|(\mathcal{I}- \mathcal{K})^{-1}\|^{-1}\|y-y_l^*\|_2 \leqslant\|e_l^*\|_2 \leqslant \|\mathcal{I}- \mathcal{K}\|\|y-y_l^*\|_2,
$$
which leads to the desired estimate.
\end{proof}

Theorem~\ref{thm:bound_result2_con} establishes a two-sided bound showing that the true solution error is tightly controlled by the observable training error. As a result, convergence of the method can be reliably monitored without access to the exact solution.

The practical mechanism for improvement is driven by the optimization of the training error \( \|e_l^*\|_2 \). As shown in Proposition~\ref{thm_multi_con_e_j}, successful optimization at each grade guarantees that the training error sequence is non-increasing. Theorem~\ref{thm:bound_result2_con} then provides the essential connection between training and solution errors by bounding the latter both above and below in terms of the former. Consequently, a reduction in training error forces the admissible interval containing the solution error to shift downward, signaling genuine improvement of the approximation. The condition number \( c \) characterizes the efficiency of this transfer: values of \( c \) closer to one correspond to a tighter coupling between decreases in training error and reductions in solution error.

In summary, Theorem~\ref{the:bound_result1_con} ensures worst-case stability of the multi-grade refinement, while the monotone optimization of the training error, together with the two-sided bound of Theorem~\ref{thm:bound_result2_con}, provides a practical and theoretically justified mechanism for improving the solution approximation.

%\section{Discrete MGDL Model}
%\label{sec_multi}

\section{Error Analysis of the Discrete MGDL Model}
\label{section_5}
In this section, we establish key properties of the discrete MGDL model that parallel those of the continuous model discussed in Section~\ref{Analysis-of-Cont}. In particular, we derive lower and upper bounds on the approximation error of the discrete solution, expressed in terms of the training loss and the quadrature error. The main result quantifies how close the discrete MGDL solution is to the true solution, based on the training error—which is computable during optimization—and the quadrature error—which reflects the accuracy of the numerical integration scheme.

Following a similar approach to the continuous case,  for each $l=1,2, \dots, L$, we fix the feature $\tilde{\mathbf{g}}_l$ of grade $l$ and define the operator $\tilde{\mathcal{M}}_{l}: \Theta_{l,2}\to C(I)$ as
\begin{equation}
\tilde{\mathcal{M}}_{l}((\mathbf{W}, \mathbf{b})):=(\mathcal{I}-\mathcal{K}_{p_{_\kappa}})\mathcal{T}(\mathbf{W}\tilde{\mathbf{g}}_{l}(\cdot)+\mathbf{b})\in C(I), \ \mbox{for all}\  (\mathbf{W},\mathbf{b})\in \Theta_{l,2}. \label{def_tilde_M_l}
\end{equation}
We define the linear space $\tilde{\mathcal{M}}_{l}(\Theta_{l, 2})$ analogously to the continuous case. We also set
\begin{equation}
\tilde{z}_l:=\tilde{\mathcal{M}}_{l}\left(\left(\tilde{\mathbf{W}}^*_{l,n_{l}}, \tilde{\mathbf{b}}^*_{l,n_{l}}\right)\right)\in \tilde{\mathcal{M}}_{l}(\Theta_{l, 2}). \label{def_tilde_z_l}
\end{equation}
Then, in direct analogy with Lemma \ref{lem_multi3_feature_space} for the continuous MGDL model, we immediately obtain the following lemma.

\begin{lemma} \label{lem_feature_space_dis}
For any $l=1,2, \dots, {L-1}$, $\tilde{z}_{l+1}$ is the best approximation to $\tilde{e}_l^*$ from the linear space $\tilde{\mathcal{M}}_{l+1}(\Theta_{l+1, 2})$ in the sense of the semi-norm $\|\cdot\|_{N}$.
\end{lemma}

\begin{proof}
Fix any $l=1,2, \dots, {L-1}$. We apply Lemma~\ref{lem_mult1} with the following settings:
\[
\mathcal{B}_1 := \Theta_{l+1,1}, \quad \mathcal{B}_2 := \Theta_{l+1,2},
\]
\[
\theta_1^* := \left(\{\tilde{\mathbf{W}}_{l+1,j}^*, \tilde{\mathbf{b}}_{l+1,j}^*\}_{j=1}^{n_{l+1}-1}\right) \in \Theta_{l+1,1},
\]
\[
\theta_2^* := \left( \tilde{\mathbf{W}}_{l+1,n_{l+1}}^*, \tilde{\mathbf{b}}_{l+1,n_{l+1}}^* \right) \in \Theta_{l+1,2},
\]
and the objective function $F$ defined in equation \eqref{mult_weight_dis}. 

By Lemma~\ref{lem_mult1}, the pair $(\theta_1^*, \theta_2^*)$ implies that $\theta_2^*$ is a local minimizer of the function
$\theta \mapsto \left\|\tilde{e}_{l+1}\left(\theta_1^* \times \theta ;\cdot\right) \right\|_{N}^2$.
From the definition of $\tilde{e}_{l+1}$ in equation \eqref{def_e_q+1_dis}, this minimization is equivalent to
\[
\min_{\theta \in \Theta_{l+1,2}} \left\| \tilde{e}_l^* - \tilde{\mathcal{M}}_{l+1}(\theta) \right\|_N^2.
\]

Now apply Lemma~\ref{lem_multi2} with:
\[
\mathcal{X} := C(I), \quad |\cdot| := \|\cdot\|_N, \quad \mathcal{U} := \tilde{\mathcal{M}}_{l+1}, \quad \mathcal{B} := \Theta_{l+1,2}, \quad s := \tilde{e}_l^*, \quad t^* := \theta_2^*.
\]
By Lemma~\ref{lem_multi2}, the element
$\tilde{z}_{l+1} := \mathcal{U} t^* = \tilde{\mathcal{M}}_{l+1}(\theta_2^*)$
is the best approximation to $\tilde{e}_l^*$ from the linear space $\tilde{\mathcal{M}}_{l+1}(\Theta_{l+1,2})$ with respect to the semi-norm $\|\cdot\|_N$.
\end{proof}

Unlike the continuous case, the discrete best approximation defined in Lemma \ref{lem_feature_space_dis} is not necessarily unique. Nonetheless, we can establish a similar, crucial property. To do so, we first present the following lemma, which essentially establishes the strict convexity of the discrete semi-norm $\|\cdot\|_N$.

\begin{lemma} \label{prop_strictly_convex}
For any $\phi,\psi\in C(I)$ satisfying $\|\phi\|_{N}=1, \|\psi\|_{N}=1, \|\phi-\psi\|_{N}\neq 0$,  it holds that
$$
\|x\phi+(1-x)\psi\|_{N}<1,  \ \  \mbox{for any} \ \ x\in (0,1).
$$
\end{lemma}
\begin{proof}
Define $\mathbf{v}_\phi := [\phi(x_j)]_{j=1}^N$ and $\mathbf{v}_\psi := [\psi(x_j)]_{j=1}^N$ in $\mathbb{C}^N$. From the definition of $\|\cdot\|_N$ and relation \eqref{relation_N_2}, we have
\[
\|\phi\|_N = \|\psi\|_N = 1 \quad \Rightarrow \quad \|\mathbf{v}_\phi\|_2 = \|\mathbf{v}_\psi\|_2 = \sqrt{N},
\]
and $\|\phi - \psi\|_N \neq 0 \Rightarrow \mathbf{v}_\phi \neq \mathbf{v}_\psi$.

The $\ell_2$-norm in $\mathbb{C}^N$ is strictly convex, so for all $x \in (0,1)$,
\[
\left\|x \cdot \frac{\mathbf{v}_\phi}{\sqrt{N}} + (1 - x) \cdot \frac{\mathbf{v}_\psi}{\sqrt{N}} \right\|_2 < 1.
\]
By linearity and the definition of $\|\cdot\|_N$, we then have
\[
\|x\phi + (1 - x)\psi\|_N = \frac{\|x\mathbf{v}_\phi + (1 - x)\mathbf{v}_\psi\|_2}{\sqrt{N}} < 1.
\]
This proves the claim.
\end{proof}

In contrast to the continuous model, where the best approximation is unique, the discrete model satisfies the following alternative property.

\begin{lemma} \label{lem_unique}
For any $l\in \mathbb{N}_{L-1}$, let $\phi \in \tilde{\mathcal{M}}_{l+1}(\Theta_{l+1, 2})$ be a best approximation to $\tilde{e}_l^*$ from the linear space $\tilde{\mathcal{M}}_{l+1}(\Theta_{l+1, 2})$ with respect to the semi-norm $\|\cdot\|_{N}$. Then 
$$
\|\phi-\tilde{z}_{l+1}\|_{N}=0.
$$
\end{lemma}

\begin{proof}
Define
$d := \|\phi - \tilde{e}_l^*\|_N$. By Lemma~\ref{lem_feature_space_dis}, we also have \( \|\tilde{z}_{l+1} - \tilde{e}_l^*\|_N = d \), and for all \( h \in \tilde{\mathcal{M}}_{l+1}(\Theta_{l+1,2}) \),
\begin{equation}
    \|h - \tilde{e}_l^*\|_N \geq d. \label{lem_unique_equ1}
\end{equation}

We consider two cases:

Case 1: \( d = 0 \). Then by the triangle inequality of the
semi-norm $\|\cdot\|_N$, we obtain
\[
\| \phi -\tilde{z}_{l+1}\|_N\leq \|\phi - \tilde{e}_l^*\|_N+\|\tilde{e}_l^*-\tilde{z}_{l+1}\|_N = 2d=0.
\]

Case 2: \( d \neq 0 \). 
Since both \( \phi \) and \( \tilde{z}_{l+1} \) belong to the linear space \( \tilde{\mathcal{M}}_{l+1}(\Theta_{l+1,2}) \), their convex combination also lies in this space. For any \( \mu \in (0,1) \), define
\[
h := \mu \phi + (1 - \mu) \tilde{z}_{l+1} \in \tilde{\mathcal{M}}_{l+1}(\Theta_{l+1,2}).
\]
Using \eqref{lem_unique_equ1}, we have
\begin{equation}
    \|h - \tilde{e}_l^*\|_N = \|\mu\phi + (1-\mu)\tilde{z}_{l+1} - \tilde{e}_l^*\|_N \geq d. \label{lem_unique_equ2}
\end{equation}

Define normalized vectors:
\[
\psi_1 := \frac{\phi - \tilde{e}_l^*}{d}, \quad \psi_2 := \frac{\tilde{z}_{l+1} - \tilde{e}_l^*}{d},
\]
so that \( \|\psi_1\|_N = \|\psi_2\|_N = 1 \). We then observe 
$$
\frac{1}{d}[(\mu\phi+(1-\mu)\tilde{z}_{l+1})-\tilde{e}_l^*]=\mu\psi_1+(1-\mu)\psi_2, \ \ \mbox{for all}\ \ \mu\in (0,1).
$$
Using the equation above, inequality \eqref{lem_unique_equ2} becomes:
\[
\|\mu \psi_1 + (1 - \mu) \psi_2\|_N \geq 1.
\]
Moreover, we have $\|\psi_1-\psi_2\|_N\neq 0$ since $d\neq 0$. This contradicts the strict convexity of the semi-norm \( \|\cdot\|_N \) (Lemma~\ref{prop_strictly_convex}), which implies:
\[
\|\mu \psi_1 + (1 - \mu) \psi_2\|_N < 1, \quad \text{for all } \mu \in (0,1).
\]
Thus, the desired result follows.
\end{proof}

At this point, we are ready to prove that the sequence $\|\tilde{e}_l^*\|_{N}$, for $l=1,2,\dots, L$, is non-increasing.

\begin{proposition}\label{thm_e_j}
For each for $l=1,2,\dots, L-1$, the optimal error satisfies
\[
\|\tilde{e}^*_{l+1}\|_{N} \leq \|\tilde{e}^*_{l}\|_{N}.
\]
Moreover, equality holds if and only if $\|\tilde{z}_{l+1}\|_{N} = 0$; that is,
\[
\|\tilde{e}^*_{l+1}\|_{N} = \|\tilde{e}^*_{l}\|_{N} \quad \Longleftrightarrow \quad \|\tilde{z}_{l+1}\|_{N} = 0.
\]
\end{proposition}

\begin{proof}
Fix $l=1,2,\dots, L-1$. From the definition of $\tilde{e}_{l+1}^*$ in \eqref{def_fin_e_q+1_dis} and of $\tilde{z}_{l+1}$ in \eqref{def_tilde_z_l}, we have
\[
\tilde{e}_{l+1}^* = \tilde{e}_l^* - \tilde{z}_{l+1}.
\]
By Lemma~\ref{lem_feature_space_dis}, $\tilde{z}_{l+1}$ is the best approximation to $\tilde{e}_l^*$ from the space $\tilde{\mathcal{M}}_{l+1}(\Theta_{l+1, 2})$. Since the zero function belongs to this space, we obtain
\[
\|\tilde{e}_{l+1}^*\|_N = \|\tilde{e}_l^* - \tilde{z}_{l+1}\|_N \leq \|\tilde{e}_l^* - 0\|_N = \|\tilde{e}_l^*\|_N,
\]
which proves the first assertion.

To prove the second statement, suppose
\[
\|\tilde{e}_{l+1}^*\|_N = \|\tilde{e}_l^*\|_N.
\]
Then from the inequality above, equality must hold:
\[
\|\tilde{e}_l^* - \tilde{z}_{l+1}\|_N = \|\tilde{e}_l^* - 0\|_N.
\]
This means that both $\tilde{z}_{l+1}$ and $0$ are best approximations to $\tilde{e}_l^*$ in $\tilde{\mathcal{M}}_{l+1}(\Theta_{l+1,2})$. By Lemma~\ref{lem_unique}, their difference in the semi-norm $\|\cdot\|_N$ must be zero, so we have
\[
\|\tilde{z}_{l+1}\|_N =\|\tilde{z}_{l+1} - 0\|_N =  0.
\]
\end{proof}

Proposition~\ref{thm_e_j} extends Theorem 5 of \cite{Xu:2023aa}, originally formulated for function approximation, to the setting of numerically solving integral equations.
The second part of Proposition~\ref{thm_e_j} further reveals that if $\|\tilde{z}_{l+1}\|_N\neq 0$, then the training error \(\|\tilde{e}_{l+1}^*\|_N\) must strictly decrease. In other words, grade \(l+1\) contributes meaningfully to improving the solution. Conversely, if  \(\|\tilde{z}_{l+1}\|_N = 0\), then no further improvement is possible within the approximation space \(\tilde{\mathcal{M}}_{l+1}(\Theta_{l+1,2})\). 

\begin{remark}
This gives a practical diagnostic criterion: \emph{the effectiveness of the added grade can be judged by whether \(\tilde{z}_{l+1}\) is nonzero}. If the improvement is negligible (i.e., \(\|\tilde{z}_{l+1}\|_N \approx 0\)), then deeper grading may be unnecessary, and training can be stopped early.
\end{remark}

Next, we define the discrete MGDL approximation at grade $l$ by
\begin{equation}
    \tilde{y}^*_{l}:=\sum_{j=1}^{l} \mathcal{T} \tilde{\mathbf{f}}_j \in C(I), \quad l=1,2,\dots, L. \label{def_y_l_dis}
\end{equation}

The following lemma establishes a representation of the optimal error $\tilde{e}_l^*$ in terms of the true solution $y$ and the discrete approximation $\tilde{y}_l^*$.

\begin{lemma} \label{lem_tilde_e_l_star_dis}
For each $l \in \mathbb{N}_L$, the optimal error satisfies
\[
\tilde{e}_l^* = (\mathcal{I}- \mathcal{K}_{p_\kappa})(y - \tilde{y}_{l}^*) + (\mathcal{K}_{p_\kappa} - \mathcal{K})y.
\]
\end{lemma}

\begin{proof}
From the definition of $\tilde{e}_1^*$ in equation \eqref{def_fin_e1_dis}, we have
\begin{equation}
\tilde{e}_1^* = f - (\mathcal{I} -  \mathcal{K}_{p_\kappa})\mathcal{T} \tilde{\mathbf{f}}_1 = (\mathcal{I} -  \mathcal{K}_{p_\kappa})(y - \mathcal{T} \tilde{\mathbf{f}}_1) + (\mathcal{K}_{p_\kappa} - \mathcal{K})y. \label{e_1_star_equ1}
\end{equation}

For $l=1,2, \dots, L-1$, the recursive definition of $\tilde{e}_{l+1}^*$ from equation \eqref{def_fin_e_q+1_dis} gives
\[
\tilde{e}_{l+1}^* = \tilde{e}_l^* - (\mathcal{I} -  \mathcal{K}_{p_\kappa})\mathcal{T} \tilde{\mathbf{f}}_{l+1}.
\]
Applying this recursively and using equation \eqref{e_1_star_equ1}, we obtain
\[
\tilde{e}_l^* = (\mathcal{I} -  \mathcal{K}_{p_\kappa})\left(y - \sum_{j=1}^l \mathcal{T} \tilde{\mathbf{f}}_j\right) + (\mathcal{K}_{p_\kappa} - \mathcal{K})y.
\]
By the definition of $\tilde{y}_l^*$ in equation \eqref{def_y_l_dis}, this simplifies to the desired expression.
\end{proof}

We now present the discrete analogue of Theorem~\ref{the:bound_result1_con}, returning to the discrete system \eqref{discrete_operator_Equation}. When the quadrature nodes coincide with the collocation points (in this case, $p_\kappa = N-1$) and $\mathcal{K}_{p_\kappa}$ converges pointwise to $\mathcal{K}$, there exists, for sufficiently large $N$, an invertible matrix $\mathbf{M}_\kappa \in \mathbb{C}^{N \times N}$ satisfying
\begin{equation}\label{Discrete-Repsentationn}
\left[((\mathcal{I} - \mathcal{K}_{p_\kappa})h)(x_j) \right]_{j=1}^N = \mathbf{M}_\kappa \mathbf{v}_h,
\end{equation}
for any function $h \in C(I)$, where $\mathbf{v}_h := [h(x_j)]_{j=1}^N$.

Accordingly, the discrete integral equation \eqref{discrete_operator_Equation} reduces to the linear system
\begin{equation}\label{equ_dis}
\mathbf{M}_\kappa \mathbf{v}_y = \mathbf{v}_f,
\end{equation}
where $\mathbf{v}_f := [f(x_j)]_{j=1}^N$ and $\mathbf{v}_y := [y(x_j)]_{j=1}^N$.

% In the general case where the quadrature nodes do not necessarily coincide with the collocation points, \cite{jiang2024deep} shows that $\mathbf{M}_\kappa$ is invertible for almost every $\kappa > 0$. We henceforth assume invertibility without further comment.

%Consequently, when the continuous integral equation admits a unique solution, the discretized equation \eqref{equ_dis} will also admit a unique solution \cite{anselone1971collectively} for sufficiently large $p_\kappa$.

%Correspondingly, there exists a matrix $\mathbf{M}_\kappa\in \mathbb{C}^{N\times N}$ such that for all $\kappa\geqslant 1$ and for any function $h\in C(I)$, the following relation holds:
%\begin{equation}
%    [((\mathcal{I}-\mathcal{K}_{p_{_\kappa}})h)(x_j): j=1,2, \dots, N]^T=\mathbf{M}_\kappa \mathbf{v}_h,\label{dis_con_rel}
%\end{equation}
%where $\mathbf{v}_h:=[h(x_j)]_{j=1}^N$. 

%%%%%%%%%%%%%%%%%%

\begin{lemma} \label{dis_inv}
For any $h \in C(I)$, the following bounds hold:
\[
\|\mathbf{M}_\kappa^{-1}\|_2^{-1} \|h\|_N \leqslant \|(\mathcal{I} -  \mathcal{K}_{p_\kappa})h\|_N \leqslant \|\mathbf{M}_\kappa\|_2 \|h\|_N.
\]
\end{lemma}

\begin{proof}
For any $h \in C(I)$, let $\mathbf{v}_h := [h(x_j) : j \in \mathbb{N}_N]^T \in \mathbb{C}^N$. By relation~\eqref{relation_N_2} and the definition of $\mathbf{M}_\kappa$, we have
\begin{equation} \label{dis_inv_equ1}
    \|h\|_N = \frac{\|\mathbf{v}_h\|_{\ell_2}}{\sqrt{N}}, \qquad \|(\mathcal{I} -  \mathcal{K}_{p_\kappa})h\|_N = \frac{\|\mathbf{M}_\kappa \mathbf{v}_h\|_{\ell_2}}{\sqrt{N}}.
\end{equation}
Using the submultiplicativity and invertibility properties of the spectral norm, we obtain
\[
\|\mathbf{M}_\kappa^{-1}\|_2^{-1} \|\mathbf{v}_h\|_{\ell_2} \leqslant \|\mathbf{M}_\kappa \mathbf{v}_h\|_{\ell_2} \leqslant \|\mathbf{M}_\kappa\|_2 \|\mathbf{v}_h\|_{\ell_2}.
\]
Dividing all sides by $\sqrt{N}$ and applying~\eqref{dis_inv_equ1}, we obtain the desired estimate.
\end{proof}

We now present the discrete analogue of Theorem~\ref{the:bound_result1_con}.

\begin{theorem} \label{prop_error_multi_discrete}
For $l \in \mathbb{N}_{L-1}$, the following estimate holds:
\[
\|y - \tilde{y}_{l+1}^*\|_N \leq \mathrm{cond}(\mathbf{M}_\kappa)\|y - \tilde{y}_l^*\|_N + 2\|\mathbf{M}_\kappa^{-1}\|_2 \|(\mathcal{K}_{p_\kappa} - \mathcal{K})y\|_N,
\]
where $\mathrm{cond}(\mathbf{M}_\kappa)$ denotes the condition number of the matrix $\mathbf{M}_\kappa$.
\end{theorem}

\begin{proof}
By Lemma~\ref{lem_tilde_e_l_star_dis}, for all $l \in \mathbb{N}_{L-1}$, we have
\[
\tilde{e}_l^* = (\mathcal{I} -  \mathcal{K}_{p_\kappa})(y - \tilde{y}_l^*) + (\mathcal{K}_{p_\kappa} - \mathcal{K})y,
\]
\[
\tilde{e}_{l+1}^* = (\mathcal{I} -  \mathcal{K}_{p_\kappa})(y - \tilde{y}_{l+1}^*) + (\mathcal{K}_{p_\kappa} - \mathcal{K})y.
\]
Substituting into the inequality $\|\tilde{e}_{l+1}^*\|_N \leq \|\tilde{e}_l^*\|_N$ (from Proposition~\ref{thm_e_j}) gives
\[
\|(\mathcal{I} -  \mathcal{K}_{p_\kappa})(y - \tilde{y}_{l+1}^*) + (\mathcal{K}_{p_\kappa} - \mathcal{K})y\|_N
\leq
\|(\mathcal{I} -  \mathcal{K}_{p_\kappa})(y - \tilde{y}_l^*) + (\mathcal{K}_{p_\kappa} - \mathcal{K})y\|_N.
\]
Applying the triangle inequality yields
\[
\|(\mathcal{I} -  \mathcal{K}_{p_\kappa})(y - \tilde{y}_{l+1}^*)\|_N \leq \|(\mathcal{I} -  \mathcal{K}_{p_\kappa})(y - \tilde{y}_l^*)\|_N + 2\|(\mathcal{K}_{p_\kappa} - \mathcal{K})y\|_N.
\]
Using Lemma~\ref{dis_inv} to bound the operator norm in terms of the matrix $\mathbf{M}_\kappa$, we obtain
\[
\|\mathbf{M}_\kappa^{-1}\|_2^{-1} \|y - \tilde{y}_{l+1}^*\|_N \leq \|\mathbf{M}_\kappa\|_2 \|y - \tilde{y}_l^*\|_N + 2\|(\mathcal{K}_{p_\kappa} - \mathcal{K})y\|_N.
\]
Multiplying both sides by $\|\mathbf{M}_\kappa^{-1}\|_2$ yields the result.
\end{proof}

Theorem~\ref{prop_error_multi_discrete} establishes that in the discrete (implementable) setting, the error at grade $l+1$ is bounded by the error at grade 
$l$, scaled by the condition number of the system matrix $\mathbf{M}_\kappa$, plus a term reflecting the quadrature error, that is, the accuracy of approximating $\mathcal{K}$ by $\mathcal{K}_{p_\kappa}$. 
This mirrors the continuous case and confirms that the practical MGDL algorithm retains its error-reduction behavior when discretized. When  $\mathbf{M}_\kappa$ is well-conditioned, as is typically the case \cite{Atkinson}, and the quadrature error is small, the approximation error decreases reliably with each added grade.

In the next section, we show that the quadrature error 
\begin{equation}\label{quadrature_error}
    R_{p_\kappa}:=\|(\mathcal{K}_{p_\kappa}-\mathcal{K})y\|_{N}
\end{equation}
can be made arbitrarily small by choosing sufficiently large $p_\kappa$. Therefore, under suitable conditions, the discrete MGDL model behaves similarly to its continuous counterpart in Theorem~\ref{the:bound_result1_con}.

The following theorem shows that minimizing the training error leads to a corresponding reduction in the solution error.

\begin{theorem} \label{prop_error_multi_discrete2}
For all $l\in \mathbb{N}_L$, the following estimate holds:
\begin{equation}
\|\mathbf{M}_\kappa\|_2^{-1} (\|\tilde{e}_l^*\|_{N}-R_{p_\kappa})\leqslant \|y-\tilde{y}_{l}^*\|_{N}\leqslant \|\mathbf{M}_\kappa^{-1}\|_2 (\|\tilde{e}_l^*\|_{N}+R_{p_\kappa}),\label{equ1}
\end{equation}
where  $R_{p_\kappa}:=\|(\mathcal{K}_{p_\kappa}-\mathcal{K})y\|_{N}$ is the quadrature error defined in \eqref{quadrature_error}.
\end{theorem}
\begin{proof}
We first prove the upper bound in \eqref{equ1}. By Lemma~\ref{lem_tilde_e_l_star_dis}, we have
$$
(\mathcal{I}- \mathcal{K}_{p_\kappa})(y-\tilde{y}_{l}^*)=\tilde{e}_l^*-(\mathcal{K}_{p_\kappa}-\mathcal{K})y.
$$
Applying the triangle inequality gives
$$
\|(\mathcal{I}- \mathcal{K}_{p_\kappa})(y-\tilde{y}_{l}^*)\|_N\leqslant \|\tilde{e}_l^*\|_N+R_{p_\kappa}.  
$$
Using Lemma~\ref{dis_inv}, we obtain
$$
\|y-\tilde{y}_l^*\|_N\leq \|\mathbf{M}_\kappa^{-1}\|_2(\|\tilde{e}_l^*\|_N+R_{p_\kappa}),
$$
which proves the upper bound.

For the lower bound, again from Lemma~\ref{lem_tilde_e_l_star_dis} and the triangle inequality,
$$
\|\tilde{e}_l^*\|_N\leqslant  \|(\mathcal{I}- \mathcal{K}_{p_\kappa})(y-\tilde{y}_{l}^*)\|_N+R_{p_\kappa}.
$$
Applying Lemma \ref{dis_inv} again gives
$$ 
\|\tilde{e}_l^*\|_N\leqslant  \|\mathbf{M}_\kappa\|_2\|y-\tilde{y}_{l}^*\|_N+R_{p_\kappa},  
$$
which rearranges to the desired lower bound:
$$
\|y-\tilde{y}_l^*\|_N\geq \|\mathbf{M}_\kappa\|_2^{-1}(\|\tilde{e}_l^*\|_N-R_{p_\kappa}).
$$
\end{proof}

Theorem~\ref{prop_error_multi_discrete2} provides a key theoretical basis for the adaptive MGDL algorithm introduced in the next section. It shows that the solution error is tightly bounded above and below by the training error and the quadrature error. As a result, minimizing the training error reliably reduces the true solution error, assuming the quadrature error is small. This justifies using the training error as a proxy for solution accuracy, which is critical for adaptively increasing the number of grades in the MGDL model.

\section{Adaptive MGDL Algorithm}
\label{sec_adaptive}

In this section, we introduce the adaptive MGDL (AMGDL) algorithm for approximating the solution of the equation within a prescribed tolerance. Each grade consists of a shallow network with a single hidden layer. Unlike in Section~2, we do not predefine the maximum number $L$ of grades; instead, AMGDL automatically selects an optimal number. The integral operator $\mathcal{K}$ is approximated using the composite trapezoidal rule. We establish a theoretical result showing that the solution produced by AMGDL is guaranteed to lie within the specified tolerance, with computational time scaling linearly in the number of grades.

We begin by recalling the discrete oscillatory integral operator $\mathcal{K}_{p_\kappa}$ from our previous work~\cite{jiang2024deep}, which provides a numerical approximation of $\mathcal{K}$. The number of quadrature nodes is defined by
$p_\kappa := \lceil \gamma \kappa^\beta \rceil, \label{def_p_kappa}$
with parameters $\gamma$ and $\beta$ satisfying
\begin{equation}
\Gamma \geqslant 0, \quad \beta \geqslant 1, \quad \gamma \geqslant \Gamma + 3. \label{rule_beta_gamma}
\end{equation}
Using the composite trapezoidal rule, we define the discrete operator $\mathcal{K}_{p_\kappa}$ as follows: For any $F \in C(I)$ and $s \in I$,
\begin{align*}
    (\mathcal{K}_{p_\kappa} F)(s) := &\frac{h}{2} \left[ F(s_0) K(s, s_0)e^{i\kappa|s - s_0|} + 2 \sum_{j=0}^{p_\kappa} K(s,s_j)F(s_j) e^{i\kappa|s - s_j|}\right. \\
    & \left. + K(s, s_{p_\kappa}) F(s_{p_\kappa}) e^{i\kappa|s - s_{p_\kappa}|} \right], \label{def_mathcal_K_p}
\end{align*}
where $h := 2/p_\kappa$ and $s_j := -1 + jh$ for $j =0, 1,  \dots, p_\kappa$. 
The training data $\{(x_j, f(x_j)) : j \in \mathbb{N}_{N_\kappa}\}$ are sampled at collocation points on a uniform grid:
\[
x_j := -1 + \frac{2(j-1)}{N_\kappa - 1}, \quad j \in \mathbb{N}_{N_\kappa}, \quad N_\kappa := p_\kappa q + 1, \quad \text{for a fixed } q \in \mathbb{N}.
\]

The matrix representation of $(\mathcal{I} -  \mathcal{K}_{p_\kappa})$ on the training grid is given by~\cite{jiang2024deep}:
\begin{equation}
    \mathbf{M}_\kappa := \mathbf{I}_\kappa - \frac{ \mathbf{B}_\kappa}{p_\kappa} \in \mathbb{C}^{N_\kappa \times N_\kappa}, \label{def_M_kappa}
\end{equation}
where $\mathbf{I}_\kappa$ is the identity matrix and $\mathbf{B}_\kappa = (b_{j,l}(\kappa))_{j,l \in \mathbb{N}_{N_\kappa}}$ is defined by
\begin{equation*}
b_{j,l}(\kappa) :=
\begin{cases}
    K(x_j, x_l)\omega_\kappa^{|j-l|}, & \text{if } l = 1 \text{ or } l = N_\kappa, \\
    2K(x_j, x_l)\omega_\kappa^{|j-l|}, & \text{if } l = dq + 1 \text{ for } d \in \mathbb{N}_{p_\kappa - 1}, \\
    0, & \text{otherwise},
\end{cases}
\end{equation*}
with $\omega_\kappa := e^{\frac{2\kappa i}{qp_\kappa}}$.

%Theorems \ref{prop_error_multi_discrete} and \ref{prop_error_multi_discrete2} hold for $\kappa\in S()$. In \cite{jiang2024deep}, for the given $\mathbf{M}_\kappa$ in equation \eqref{def_M_kappa}, the set $S()$ is defined as follows.
%For $d\geqslant 4$, let $\mathbb{C}[x]^{d\times d}$ denote the set of $d\times d$ matrices with polynomial entries in $x\in \mathbb{C}$. For $\lambda\in \mathbb{C}$, define the matrix
%$$
%\mathbf{A}_d(x;\lambda):=(a_{j,k}(x;\lambda))_{j,k\in \mathbb{N}_d}\in \mathbb{C}[x]^{d\times d},
%$$
%where
%$$
%a_{j,l}(x;\lambda):=\begin{cases}
%\lambda- (d-1), &j=l\in \{1, d\}, \\
%\lambda-\frac{d-1}{2}, & j=l\in\{2,3,\dots, d-1\},\\
%\lambda x^{|j-l|}, & \mathrm{otherwise.}
%\end{cases}
%$$
%Then
%\begin{equation*}
%    S(\lambda):=\left\{\kappa \geqslant 1:  {\rm det}\left(\mathbf{A}_{p_{_\kappa}+1}\left(e^{\frac{2 \kappa}{p_{_\kappa}}i}; \lambda\right)\right)\neq 0\right\}. \label{def_s}
%\end{equation*}
%It has been proved that $S(\lambda)$ is dense in $[1, +\infty)$ (see Proposition 5.1 in \cite{jiang2024deep}). Hence, Theorems \ref{prop_error_multi_discrete} and \ref{prop_error_multi_discrete2} hold for almost every $\kappa\in[1,+\infty)$.

Next we present an error bound on $\|(\mathcal{K}_{p_\kappa}-\mathcal{K})y\|_{N}$. Following \cite{wang2015oscillation}, there exist functions $u_j$, $j\in \mathbb{N}_3$, 
satisfying 
\begin{equation}
|u_j^{(l)}(s)| \leqslant \tau , \quad \mbox{for all}\ \ s\in I, \  \ l\in \mathbb{Z}_{m+1}, \label{con_u}
\end{equation}
for some  $\tau>0$, such that the solution $y$ of equation \eqref{fredholm_equation_operator} belongs to $H^m_{\kappa, 0}(I)$. That is, 
\begin{equation*}
    y(s) = u_1(s) + u_2(s) e^{i\kappa s} + u_3(s) e^{-i\kappa s}, \quad s\in I. \label{solution}
\end{equation*}
Equation \eqref{con_u}, together with the smoothness of the kernel $K$ (specifically in its second variable), implies the existence of a constant $\tilde{\tau} > 0$ such that, for every $t \in I$, 
\begin{equation*}
|(K(t,\cdot)u_j)^{(l)}(s)| \leqslant \tilde{\tau} , \quad \mbox{for all}\ \ s\in I, \  \ l\in \mathbb{Z}_{m+1}, 
\end{equation*}
For the discrete integral operator $\mathcal{K}_{p_\kappa}$, Proposition 4.3 in \cite{jiang2024deep} states that
\begin{equation}\label{EError}
\|\mathcal{K}y-\mathcal{K}_{p_{_\kappa}} y\|_\infty \leqslant  \mathcal{E}(\gamma, \beta, \kappa, m),
\end{equation}
where
\begin{equation}\label{def:E}
\mathcal{E}(\gamma, \beta, \kappa, m):= \frac{132\tilde{\tau}}{5\gamma \kappa^\beta}+\frac{81\tilde{\tau}(\Gamma+3)^m}{5\gamma^m \kappa^{m(\beta-1)}}. 
\end{equation}
Since 
$$
\|\mathcal{K}y-\mathcal{K}_{p_{_\kappa}} y\|_{N_\kappa}\leqslant \|\mathcal{K}y-\mathcal{K}_{p_{_\kappa}} y\|_{\infty},
$$ 
by the definition of $\|\cdot\|_{N_\kappa}$, it follows that
\begin{equation}
R_{p_\kappa}=\|\mathcal{K}y-\mathcal{K}_{p_{_\kappa}} y\|_{N_\kappa} \leqslant   \mathcal{E}(\gamma, \beta, \kappa, m), \label{integration_error}
\end{equation}
where  $R_{p_\kappa}:=\|(\mathcal{K}_{p_\kappa}-\mathcal{K})y\|_{N}$ is the quadrature error defined in \eqref{quadrature_error}.

\begin{comment}
Further for the case $|\lambda|\in (0,\frac{1}{2})$, if the parameters are chosen to satisfying 
    \begin{equation}
   \Gamma\geqslant 0,  \quad    \beta\geqslant 1,\quad  \gamma > \max\left\{\Gamma+3, \frac{4|\lambda|^2}{1-4|\lambda|^2}\right\}, \quad  q\in \left[ 1, \frac{1}{4|\lambda|^2}-\frac{1}{\lceil \gamma\rceil}\right)\cap \mathbb{N}, \label{rule_beta_gamma_q}
    \end{equation} 
then Lemma 5.8 in \cite{jiang2024deep} shows that for any $\kappa \geqslant 1$, the matrix $\mathbf{M}_\kappa \in \mathbb{C}^{N_\kappa\times N_\kappa}$ as defined in equation \cref{def_M_kappa} is invertible, this is to say $S(\lambda)=[1, +\infty)$ in this case, and
\begin{equation}
        \left\| \mathbf{M}_\kappa^{-1} \right\|_2\leqslant \frac{1}{1-\eta}, \label{lemma_inverse_target}
\end{equation}
where
     \begin{equation}
    \eta:= 2|\lambda|\sqrt{q+\frac{1}{\lceil \gamma \rceil}}\in (0,1). \label{def_eta}
   \end{equation}

Meanwhile, combining \cref{prop_error_multi_discrete2} with equations \cref{integration_error}, \cref{lemma_inverse_target} leads to following corollary.

\begin{corollary} 
For the case $\lambda\in (0,\frac{1}{2})$, $\kappa\geqslant 1$, 
\begin{equation*}
\|y-\tilde{y}^*_l\|_{N_\kappa} \leqslant \frac{1}{1-\eta}\left( \|\tilde{e}^*_l\|_{N_\kappa}+\frac{132\tau|\lambda|}{5\gamma \kappa^\beta}+\frac{81\tau|\lambda|(\Gamma+3)^m}{5\gamma^m \kappa^{m(\beta-1)}}\right), \quad l\in \mathbb{N}_L, 
\end{equation*}
where $\eta\in(0,1)$ as defined in equation \cref{def_eta} is independent of the wavenumber $\kappa$.
\end{corollary}
\end{comment}

By analyzing the definition of
$\mathcal{E}(\gamma, \beta, \kappa, m)$, it is evident that
%Firstly, by examining the equation \cref{integration_error}, it becomes evident that 
when $\beta$ and $\gamma$ are sufficiently large, the integration error $R_{p_\kappa}$ can be made arbitrarily small. Furthermore, from equation \eqref{equ1}  in Theorem \ref{prop_error_multi_discrete2}, it follows that a small integration error ensures that a small training error $\|e^*_l\|_N$ directly translates to a small overall solution error $\|y-\tilde{y}^*_l\|_N$. 

Our objective is to develop an adaptive multi-grade method that maintains the solution error $\|y-\tilde{y}^*_l\|_N$ within a predefined tolerance $\epsilon>0$. However, since this error cannot be directly observed, we instead rely on the training error $\|e^*_l\|_N$, which is directly computable. By monitoring $\|e^*_l\|_N$, we can determine an appropriate stopping criterion for the multi-grade learning process. This adaptive strategy enhances practical implementation and ensures the efficient convergence of the algorithm. The adaptive algorithm is described as follows.

\begin{algorithm}[H]
\caption{Adaptive Multi-Grade Deep Learning for Oscillatory Fredholm Integral Equation with an Error Tolerance}
\label{alg:mgdl}

\textbf{Require:} Fredholm integral equation \eqref{fredholm_equation_operator}, error tolerance $\epsilon > 0$, number of maximum grades $L\in \mathbb{N}$, parameters $\Gamma, \beta, \gamma$ satisfying condition \eqref{rule_beta_gamma} and $q\in \mathbb{N}$.

\textbf{Procedure:}
\begin{algorithmic}[1]
    \State \textbf{Initialize:} Set grade index $l = 1$. Define network $\mathcal{N}_{2}$ with parameters $\{\mathbf{W}^l_j, \mathbf{b}^l_j\}_{j=1}^{2}$.
    \State Solve the minimization problem \cref{multi_opt_problem_dis_grade_1} to obtain $\tilde{\mathbf{f}}_l$ with error $\tilde{e}_l^*$.
    \While{$l < L$ and $\|\tilde{e}_l^*\|_N > \frac{\epsilon}{2\|\mathbf{M}^{-1}_\kappa\|_2}$}
        \State Increment grade: $l \gets l+1$.
        \State Initialize network $\mathcal{N}_{2}$ with parameters $\{\mathbf{W}^l_j, \mathbf{b}^l_j\}_{j=1}^{2}$.
        \State Solve the minimization problem \eqref{mult_weight_dis} to obtain $\tilde{\mathbf{f}}_l$ with error $\|\tilde{e}_l^*\|_N$.
    \EndWhile
    \State \textbf{Return:} $\tilde{y}^*_l := \sum_{j=1}^{l} \mathcal{T} \tilde{\mathbf{f}}_j$.
\end{algorithmic}
\end{algorithm}

In the algorithm, the number $L$ differs from that in the traditional SGL model, where it denotes the parameter such that $\sum_{j=1}^L n_j$ gives the total number of hidden layers in the network. Here, $L$ instead represents the maximum number of allowed grades, and in cases of slow convergence, it also serves as part of the stopping criterion.

Note that each grade in Algorithm \ref{alg:mgdl} corresponds to a shallow network $\mathcal{N}_2$ with a single hidden layer; that is, $n_j:=2$ for all $j$. In \cite{fangXu2025}, building on the result of \cite{Pilanci2020}, it was shown that MGDL, when each grade consists of a one-hidden-layer ReLU network, reduces to solving a sequence of convex sub-problems, one for each grade. In a similar way, we can establish that the optimization problems \eqref{multi_opt_problem_dis_grade_1} and \eqref{mult_weight_dis} are equivalent to convex optimization problems when ReLU activation is employed. This reveals another important advantage of AMGDL: the resulting convex problems can be efficiently solved by gradient descent with guaranteed convergence.

The following theorem establishes that the approximate solution produced by Algorithm \ref{alg:mgdl} satisfies the prescribed error tolerance.

\begin{comment}

To analyze this algorithm, we combine \cref{prop_error_multi_discrete2} with equation \eqref{integration_error} to establish the following theorem.

\begin{theorem} \label{prop_error_adaptive}
For each $\kappa\in S(\lambda)$, if the adaptive multi-grade learning method ends at $l^*$ and if $|\lambda|\mathcal{E}(\gamma,\beta,\kappa,m)\leqslant\|\tilde{e}_l^*\|_N$, then 
\begin{equation*}
\|y-\tilde{y}_{l^*}^*\|_{N}\leqslant \epsilon.
\end{equation*}
\end{theorem}
\begin{proof}
By Theorem \ref{prop_error_multi_discrete2} and the assumptions of this theorem, the approximate solution $\tilde{y}_{l}^*$ generated by Algorithm \ref{alg:mgdl} satisfies the following estimate
\begin{equation*}
\|y-\tilde{y}_{l}^*\|_{N}\leqslant \|\mathbf{M}_\kappa^{-1}\|_2 (\|\tilde{e}_l^*\|_{N}+|\lambda|\mathcal{E}(\gamma,\beta,\kappa,m))\leqslant 2\|\mathbf{M}_\kappa^{-1}\|_2 \|\tilde{e}_l^*\|_{N}\leqslant \epsilon,
\end{equation*}
proving the desired result.
\end{proof}
\end{comment}

\begin{theorem} \label{prop_error_adaptive}
If the discrete matrix $\mathbf{M}_\kappa$ is invertible, and suppose the adaptive multi-grade deep learning algorithm terminates at iteration $l^*$ such that
\begin{equation}\label{Tolerance}
    \|\tilde{e}_{l^*}^*\|_N \leq \frac{\epsilon}{2\|\mathbf{M}^{-1}_\kappa\|_2}
\end{equation}
for a given tolerance $\epsilon>0$. If the quadrature error bound satisfies
\[
\, \mathcal{E}(\gamma, \beta, \kappa, m) \leq \|\tilde{e}_{l^*}^*\|_N,
\]
 then the resulting approximate solution $\tilde{y}_{l^*}^*$ satisfies
\[
\|y - \tilde{y}_{l^*}^*\|_N \leq \epsilon.
\]
\end{theorem}

\begin{proof}
By Theorem~\ref{prop_error_multi_discrete2}, the error between the true solution $y$ and the approximate solution $\tilde{y}_{l^*}^*$ is bounded by
\[
\|y - \tilde{y}_{l^*}^*\|_N \leq \|\mathbf{M}_\kappa^{-1}\|_2 \left( \|\tilde{e}_{l^*}^*\|_N + \, \mathcal{E}(\gamma, \beta, \kappa, m) \right).
\]
Since $ \mathcal{E}(\gamma, \beta, \kappa, m) \leq \|\tilde{e}_{l^*}^*\|_N$, we have
\[
\|y - \tilde{y}_{l^*}^*\|_N \leq 2 \|\mathbf{M}_\kappa^{-1}\|_2 \|\tilde{e}_{l^*}^*\|_N.
\]
Finally, by the stopping criterion \eqref{Tolerance}, it follows that
\[
\|y - \tilde{y}_{l^*}^*\|_N \leq \epsilon,
\]
which completes the proof.
\end{proof}

\noindent
\textbf{Remark.} In practical implementations of Algorithm~\ref{alg:mgdl}, the stopping rule based on the predefined error tolerance can be replaced with a simpler, more convenient criterion. Specifically, Proposition~\ref{thm_e_j} shows that the sequence of training errors $\|\tilde{e}_l^*\|_{N}$ is monotonically decreasing with respect to the grade index $l \in \mathbb{N}_L$. This monotonicity implies that if the training error does not decrease between two successive grades, i.e.,
\[
\|\tilde{e}_{l}^*\|_N = \|\tilde{e}_{l-1}^*\|_N,
\]
then further increasing the grade is unlikely to improve the approximation. Thus, the algorithm can be terminated when this condition is met. This alternative stopping rule is especially useful in practice, as it avoids the need to estimate or predefine the norm of the inverse matrix $\|\mathbf{M}_\kappa^{-1}\|_2$ and the error tolerance $\epsilon$.

%%%%%%%%%%%%%%
To close this section, we show that the computational cost of the AMGDL Algorithm grows linearly with the number of grades. Let $\tilde{\mathbf{f}}_l$ be a DNN of $l$ grades trained using Algorithm \ref{alg:mgdl} and assume that the widths of the grades of the network trained are bounded by a positive constant $\rho$. Denote by ${\mathcal{T}}(\tilde{\mathbf{f}}_l)$ the total training time for $\tilde{\mathbf{f}}_l$, and let $c(p)$ be the computing time required to solve the optimization problem for a network of width $p$. 

\begin{theorem}\label{Training_Time}
Suppose $\tilde{\mathbf{f}_l}$ is a DNN with $l$ grades trained by Algorithm \ref{alg:mgdl}, with grade widths bounded by a positive constant $\rho$, and let $c(p)$ denote the computing time for solving the optimization problem for a network of width $p$. Then the total training time satisfies
\begin{equation}\label{TTT}
\mathcal{T}(\tilde{\mathbf{f}}_l)\leq c(\rho)l.
\end{equation}
\end{theorem}
\begin{proof}
The total training time is the sum of the times spent on the $l$ grades:
$$
{\mathcal{T}}(\tilde{\mathbf{f}}_l)=\sum_{j=1}^lc(p_j).
$$
Since the width $p_j$ of grade $j$ ($j=1,2,\dots, l$) is bounded by a positive constant $\rho$, it follows that 
\[
c(p_j)\leq c(\rho), \quad j=1,2,\dots,l.
\]  
Therefore,  
\[
{\mathcal{T}}(\tilde{\mathbf{f}}_l)\leq \sum_{j=1}^l c(\rho)=c(\rho)l,
\]  
which proves the result.
\end{proof}

The estimate \eqref{TTT} in Theorem \ref{Training_Time} is confirmed by the numerical examples presented in the next section.

\section{Numerical Experiments}
In this section, we present two numerical examples to illustrate the effectiveness of the proposed algorithm. The first example evaluates its ability to learn solutions with multiple oscillatory scales, while the second examines its performance on solutions that not only contain multiple oscillatory scales but also feature a singularity. In both cases, the solutions involve the highest wavenumber $\kappa = 500$. For each example, we assess performance using both equal-width and varying-width networks.

An extensive comparison of DNN-based methods for solving oscillatory integral equations with traditional piecewise polynomial collocation methods was conducted in \cite{jiang2024deep}. The results show that MGDL consistently outperforms both piecewise linear and quadratic polynomial collocation methods, which are standard tools in numerical analysis. Therefore, we do not repeat that comparison here. Instead, the numerical experiments in this paper focus on comparing MGDL with the single-grade deep learning (SGDL) model and validating the theoretical results established earlier.

All experiments were conducted on an Ubuntu Server 18.04 LTS (64-bit), equipped with an Intel Xeon Platinum 8255C CPU @ 2.5GHz and an NVIDIA A4000 GPU.

We compare the performance of the  SGDL model and the AMGDL model for solving the Fredholm integral equation \eqref{fredholm_equation_operator} with a specified right-hand side. The comparison is based on the relative $L_2$ error, defined by
\begin{equation*}
\frac{\|y-\tilde{y}\|_2}{\|y\|_2}\approx \left(\frac{|y(s_0)-\tilde{y}(s_0)|^2+2\sum_{j=1}^{l-1}|y(s_j)-\tilde{y}(s_j)|^2+|y(s_l)-\tilde{y}(s_l)|^2}{|y(s_0)|^2+2\sum_{j=1}^{l-1}|y(s_j)|^2+|y(s_l)|^2}\right)^{\frac{1}{2}} \label{relative_error}
\end{equation*}
where $y$ and $\tilde{y}$ denote the exact and approximate solutions, respectively. Here, we set $l := 20,000$ and define the evaluation points by $s_j := -1 + \frac{2j}{l}$ for $j \in \mathbb{Z}_{l+1}$.

The SGDL and AMGDL models are constructed following the procedures outlined in Sections~\ref{sec_single} and~\ref{sec_adaptive}, respectively. 
For the SGDL model, the training loss is defined as
\begin{equation*}
\frac{1}{N_\kappa}\sum_{l=1}^{N_\kappa} \left|f(x_l)- ((\mathcal{I}-\mathcal{K}_{p_{_\kappa}}) \mathcal{T}\mathcal{N}_n(\{\mathbf{W}_j, \mathbf{b}_j\}_{j=1}^n);\cdot)(x_l) \right|^2,
\end{equation*}
where $x_l$ are the training data points to be specified later. 

To evaluate generalization performance, all models are assessed using a validation loss defined as
\begin{equation*}
\frac{1}{2048}\sum_{l=1}^{2048} \left|f(x'_l)- (\mathcal{I}-\mathcal{K}_{p_{_\kappa}}) Y(x'_l) \right|^2,
\end{equation*}
where $Y:=\tilde{y}^*$ for SGL and $Y:=\tilde{y}_L^*$ for AMGDL, and $x_l'$ are a set of independently selected validation points. In the implementation of Algorithm~\ref{alg:mgdl}, this validation loss replaces the training error as a stopping criterion. The only difference lies in the evaluation points, which helps reduce overfitting and improves the robustness of the adaptive selection process.

We now describe the procedure used to generate the training and validation datasets.

\textbf{Training Data:} We generate the training dataset by selecting \( N_\kappa \) equidistant points \( \{x_j\}_{j=1}^{N_\kappa} \) in the interval \( I \). For each point \( x_j \), we compute the corresponding function value \( f(x_j) \), resulting in the training dataset
\[
\{(x_j, f(x_j))\}_{j=1}^{N_\kappa} \subset I \times \mathbb{C}.
\]

\textbf{Validation Data:} The validation dataset consists of 2,048 uniformly distributed points \( \{x'_j\}_{j=1}^{2048} \) over the interval \( I \). For each \( x'_j \), we compute \( f(x'_j) \), yielding the validation dataset
\[
\{(x'_j, f(x'_j))\}_{j=1}^{2048} \subset I \times \mathbb{C}.
\]

In all experiments, the activation function used in the hidden layers is consistently chosen as the sine function for both the SGDL and AMGDL models.

\vspace{10pt}

\noindent\textbf{Example 1:}  
This example is designed to evaluate the effectiveness of the proposed methods in learning solutions that contain multiple oscillatory scales, with the highest wavenumber $\kappa = 500$.

We consider the oscillatory Fredholm integral equation \eqref{fredholm_equation_operator} with kernel 
$$
K(s,t) = \cos(s(t+1)),\ \  s,t\in I
$$ 
and $\kappa = 500$. Here, the smooth kernel $K$ is non-separable. The exact solution is chosen as
\begin{align*}
    y(s) :=\ &e^s + (\sin(s)+s+1)e^{100is} + (\cos(s)+s^3)e^{-150i s} + |s|e^{-200is} \\
    &+ s^3e^{250is} + \cosh(s) e^{300is} + s^2e^{-350is} + \sinh(s)e^{400is} \\
    &+ (s^3+\sin(s))e^{450is} + |s|^3 e^{-500is}, \quad s\in I.
\end{align*}
Although the function $y$ does not strictly belong to the space $H^m_{\kappa,0}(I)$—due to its components oscillating at different frequencies—this design introduces richer multiscale features, making the learning task more challenging. Nonetheless, the deviation from $H^m_{\kappa,0}(I)$ does not significantly affect the theoretical analysis of the model's error. The right-hand side function $f$ in \eqref{fredholm_equation_operator} is computed accordingly based on this exact solution.

For the deep neural network implementation, we set $\Gamma = 2$, $\gamma = 8$, $\beta = 1$, and $q = 1$, in compliance with the condition \eqref{rule_beta_gamma}. As a result, we have $p_\kappa := \lceil \gamma \kappa^\beta \rceil = 4,000$.

The AMGDL algorithm allows for flexible neural network design, enabling various architectural variants. Among these, we explore two representative configurations: networks with uniform widths and those with variable widths.

\noindent\textbf{Equal-Width Networks:}  
In this configuration, we set the maximum number of grades to 12 and the minimum to 3. For each grade $j \in \mathbb{N}_{12}$, the neural network architecture is defined as follows:
\begin{center}
\text{Grade $j$:} \quad $[1] \to (j-1) \times [256]_F \to [256] \to [2]$,
\end{center}
where $(j-1) \times [256]_F$ denotes $j-1$ hidden layers with 256 neurons each, whose parameters are fixed (i.e., not updated) from previous grades. Each successive grade thus adds one new trainable layer of width 256 to the network, building upon the structure from the prior grade.

\medskip
We now describe the training and hyperparameter tuning strategies for AMGDL. Each model was trained for 4{,}000 epochs using the Adam optimizer with an exponentially decaying learning rate, starting from a selected initial value and decreasing to \(10^{-7}\). The initial learning rate was chosen from \(\{10^{-1}, 10^{-2}, 10^{-3}, 10^{-4}\}\), and the batch size from \(\{128, 256, 512\}\). The optimal hyperparameter combination was determined based on the best validation performance, averaged over three independent runs. Based on this criterion, an initial learning rate of \(10^{-2}\) and a batch size of 128 were selected.

%We now describe the training and hyperparameter tuning strategies for AMGDL. {\color{red} Each model was trained for 4,000 epochs using the Adam optimizer with an exponentially decaying learning rate, starting from an initial value and decreasing to $10^{-7}$.} The initial learning rate was selected from the set $\{10^{-1}, 10^{-2}, 10^{-3}, 10^{-4}\}$, and the batch size from $\{128, 256, 512\}$. The optimal combination of hyperparameters was determined based on the best validation performance averaged over three independent runs. The final choice was an initial learning rate of $10^{-2}$ and a batch size of 128.

\begin{table}[!ht]
\centering
\scalebox{0.9}{
\begin{tabular}{c||c|c|c|c}
\hline
 & Training loss & Validation loss & RE & Time (second) \\
\hline
Grade 1 & 3.73e+0 & 3.26e+0 & 8.86e-1 & 371 \\
Grade 2 & 1.51e+0 & 1.70e+0 & 6.53e-1 & 397\\
Grade 3 & 2.15e-7 & 2.77e-7 & 2.70e-4 & 389\\
Grade 4 & 5.88e-9 & 8.24e-9 & 8.91e-5 & 394 \\
Grade 5 & 3.04e-9 & 6.45e-9 & {\bf 8.67e-5} & 392 \\
Grade 6 & 2.88e-9 & {\bf 6.39e-9} & 8.69e-5 & 406 \\
Grade 7 & 2.84e-9 & 6.84e-9 & 8.96e-5 & 418 \\
\hline
\end{tabular}}
\caption{Performance metrics across different grades for AMGDL using equal-width networks under optimal hyperparameters.}
\label{tab:exp1_amgdl1}
\end{table}

The validation errors for each grade are reported in Table~\ref{tab:exp1_amgdl1} and visualized in Figure~\ref{fig:exp1_grade1}. As observed, the validation error decreases monotonically from grade 1 to grade 6 but increases at grade 7. Hence, according to the stopping criterion of the adaptive algorithm, training is terminated at grade 6.

\begin{figure}[H]
\centering
 \includegraphics[width=0.8\textwidth]{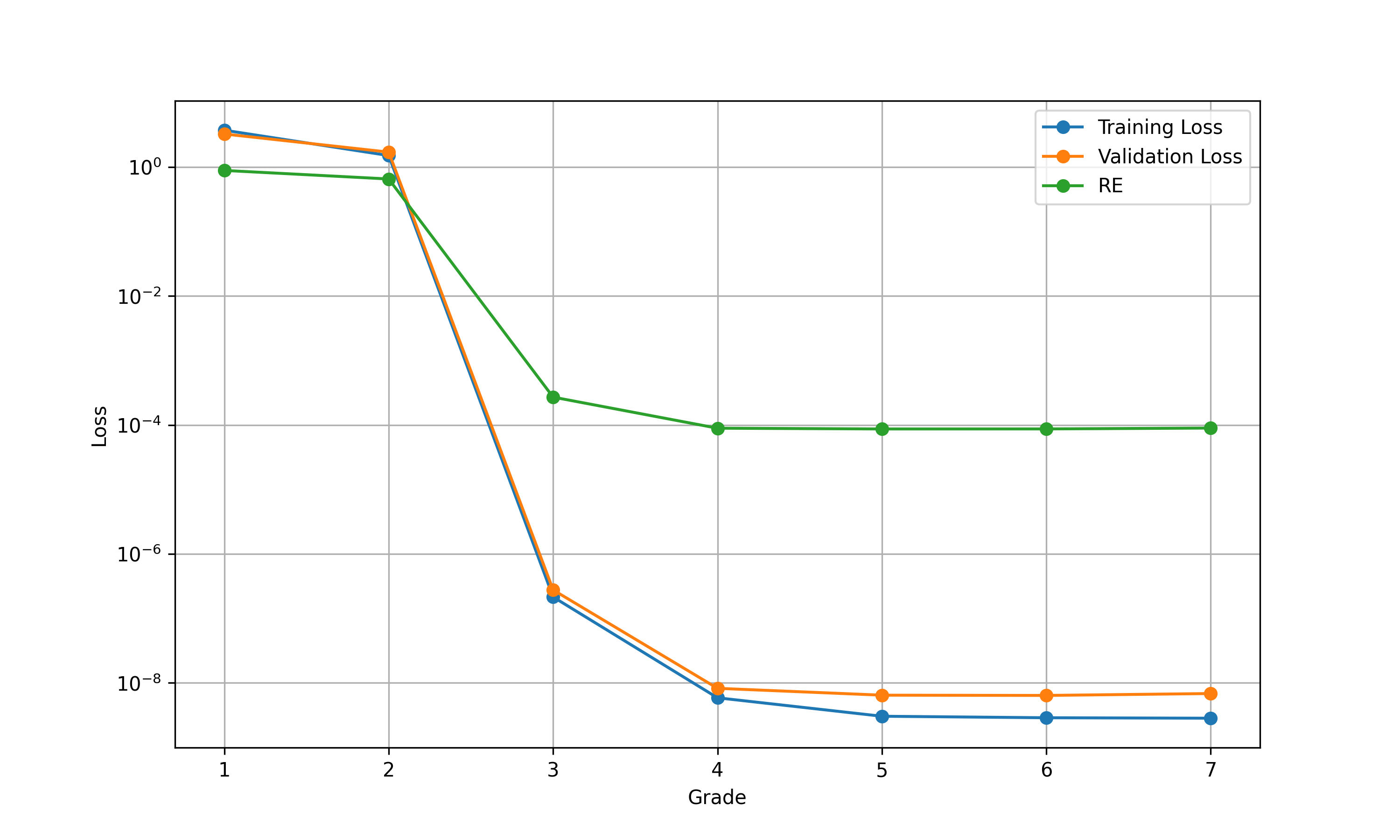}
     \caption{Training and validation losses for AMGDL using equal-width networks.}
     \label{fig:exp1_grade1}
\end{figure}

\Cref{fig:exp1_grade_freq1} shows how the AMGDL model with equal width improves approximation grade-by-grade looking from the frequency domain.
\begin{figure}[H]
\centering
 \includegraphics[width=0.6\textwidth]{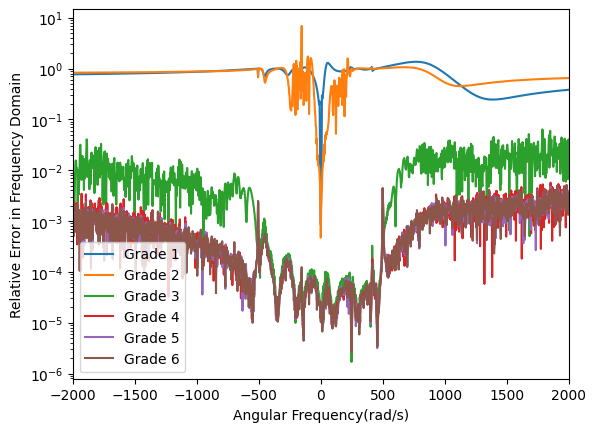}
     \caption{Frequency-domain relative errors for AMGDL  using equal-width networks: grades $1$–$7$.}
     \label{fig:exp1_grade_freq1}
\end{figure}

We compare the performance of the AMGDL model with that of the SGDL model. For each \( j \in \mathbb{N}_{10} \), the SGDL-\( j \) model uses a network architecture given by  
\[
[1] \to j \times [256] \to [2],
\]  
which corresponds to the grade-\( j \) network in the AMGDL model with equal width.

For training, each SGDL model was run for 4,000 epochs with an exponentially decaying learning rate, starting from an initial value and decreasing to \( 10^{-7} \). We tested initial learning rates from the set \( \{10^{-1}, 10^{-2}, 10^{-3}, 10^{-4}\} \) and batch sizes from \( \{128, 256, 512\} \). The best combination of hyperparameters was selected based on the lowest validation error achieved across five independent runs. The optimal hyperparameters and corresponding errors for the SGDL models are summarized in Table \ref{tab:exp1_sgl1}.

%We compare the AMGDL with the SGL model. For each $j \in \mathbb{N}_{10}$, the SGL-$j$ model has a network structure $[1] \to j\times [256] \to [2]$, which is simply the grade $j$ as in the AMGDL model with equal width. We used $4,000$ epochs and a learning rate that decreased exponentially from an initial rate to $10^{-7}$. We experimented with initial rates from the set $\{10^{-1}, 10^{-2}, 10^{-3}, 10^{-4}\}$ and batch sizes from the set $\{128, 256, 512\}$. We then chose the best pair of hyper-parameters based on the best performance across five independent experiments. After that, the best hyper-parameters and errors for the SGL models are listed in \cref{tab:exp1_sgl1}.   

\begin{table}[!ht]
\centering
\scalebox{0.9}{
\begin{tabular}{c||c|c||c|c|c}
\hline
Model & Batch size & Initial learning rate & Training loss & Validation loss & RE \\ \hline
SGDL-1 & 512 & 1e-2 & 3.24e-0 & 3.25e-0 & 8.86e-1 \\
SGDL-2 & 128 & 1e-2 & 1.42e-0 & 1.50e-0 & 6.12e-1 \\
SGDL-3 & 256 & 1e-2 & 1.14e-6 & 1.43e-6 & 5.85e-4 \\
SGDL-4 & 256 & 1e-2 & 4.77e-7 & {\bf 4.39e-7} & \textbf{3.37e-4} \\
SGDL-5 & 512 & 1e-2 & 8.19e-7 & 7.07e-7 & 4.40e-4 \\
SGDL-6 & 128 & 1e-3 & 1.27e-6 & 1.48e-6 & 6.30e-4 \\
SGDL-7 & 128 & 1e-3 & 8.02e-7 & 7.38e-7 & 4.51e-4 \\
SGDL-8 & 128 & 1e-3 & 4.34e-7 & 4.86e-7  & 3.64e-4 \\
SGDL-9 & 128 & 1e-3 & 4.13e-7 & 8.00e-7 & 5.46e-4 \\
SGDL-10 & 256 & 1e-3 & 7.65e-7 & 7.39e-7 & 4.57e-4 \\
SGDL-11 & 256 & 1e-3 & 4.81e-7 & 6.87e-7 & 5.17e-4 \\
SGDL-12 & 256 & 1e-3 & 5.46e-7 & 1.00e-6 & 7.11e-4 \\ \hline
\end{tabular}}
\caption{Optimal hyperparameters and corresponding errors for SGDL models using equal-width networks.}
\label{tab:exp1_sgl1}
\end{table}

\noindent\textbf{Varying-Width Networks:}  
We now present the example of the AMGDL algorithm using networks with varying widths. The same as in the equal-width case, we set the maximum number of grades to 12 and the minimum to 3. The network architecture at grade \( j \) is defined as follows:
\[
\text{Grade } j: \ [1] \to [a_1]_F \to [a_2]_F \to \dots \to [a_{j-1}]_F \to [a_j] \to [2], \ \text{for}\ j \in \mathbb{N}_{12},
\]
where \( a_j := 200 + 100\lceil j/2 \rceil \), for \( j \in \mathbb{N}_{12} \). Here, \([a_k]_F\) indicates that the layer of width \( a_k \) is fixed (i.e., frozen) from previous grades. Thus, at each successive grade, one new layer with increased width is added on top of the previously fixed layers.

The training and hyperparameter tuning strategy for the varying-width AMGDL model is identical to that used in the equal-width case. After evaluating multiple configurations, the optimal parameters were found to be a batch size of 256 and an initial learning rate of \( 10^{-2} \). The validation errors for each grade are reported in Table \ref{tab:exp1_amgdl2} and visualized in Figure \ref{fig:exp1_grade2}. As observed, the validation error decreases monotonically from grade 1 through grade 6, but increases at grade 7. Therefore, the algorithm terminates at grade 6.

\begin{table}[!ht]
\centering
\scalebox{0.9}{
\begin{tabular}{c||c|c|c|c}
\hline
 & Training loss & Validation loss & RE & Time (second) \\
\hline
Grade 1 & 2.99e+0 & 3.26e+0 & 8.86e-1 & 176 \\
Grade 2 & 2.23e+0 & 1.98e+0 & 7.03e-1 & 182\\
Grade 3 & 1.11e-5 & 1.09e-5 & 1.64e-3 & 206\\
Grade 4 & 9.45e-9 & 1.21e-8 & 8.76e-5 & 210 \\
Grade 5 & 6.58e-9 & 7.58e-9 & 7.84e-5 & 217 \\
Grade 6 & 2.70e-9 & \textbf{4.63e-9} & \textbf{7.48e-5} & 243 \\
Grade 7 & 2.90e-9 & 1.87e-7 & 3.26e-4 & 296 \\
\hline
\end{tabular}}
\caption{Validation errors across different grades for AMGDL using equal-width networks under optimal hyperparameters.}
\label{tab:exp1_amgdl2}
\end{table}

\begin{figure}[H]
\centering
 \includegraphics[width=0.5\textwidth]{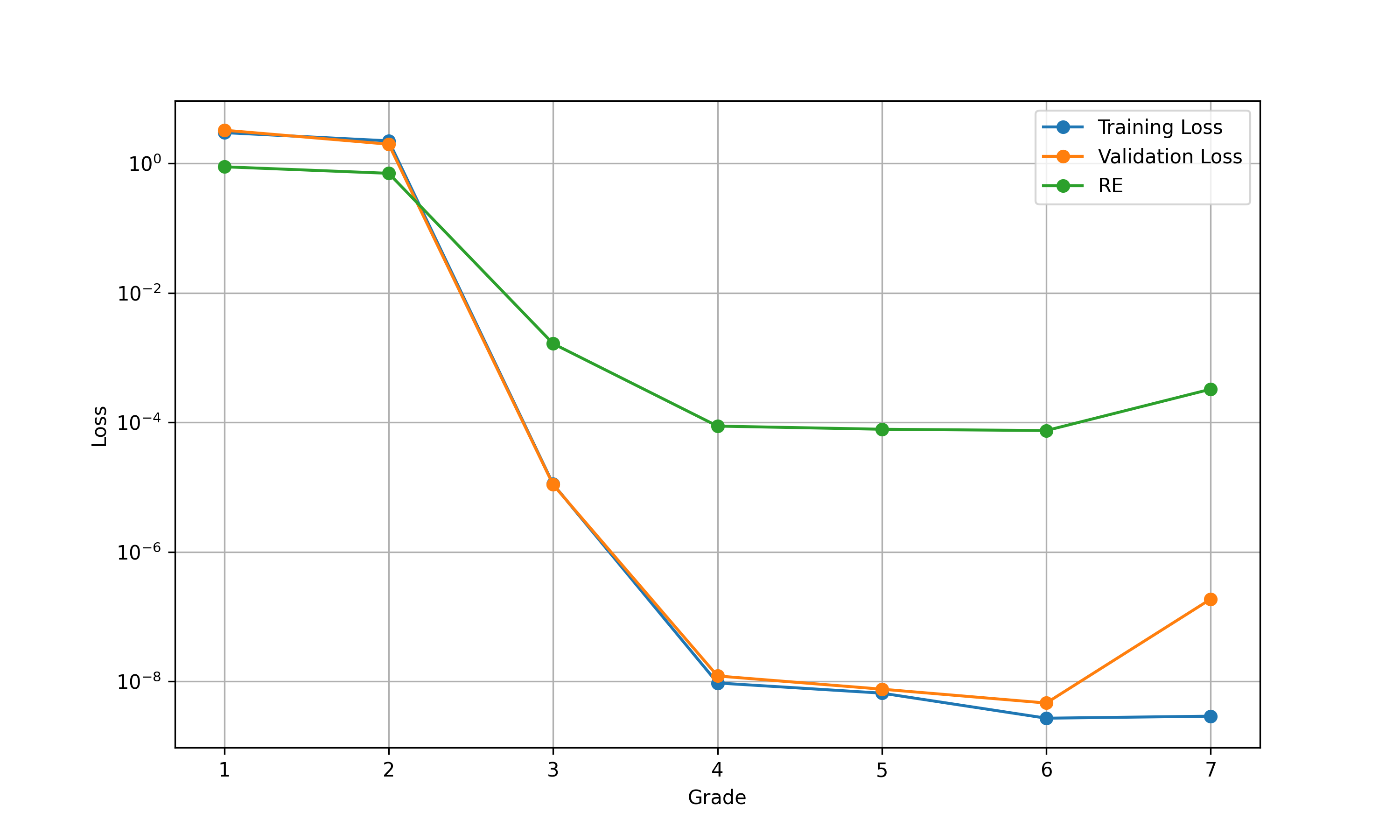}
     \caption{Training and validation losses for AMGDL using varying-width networks.}
     \label{fig:exp1_grade2}
\end{figure}

Next, we illustrate how the MGDL model improves approximation quality across grades in the frequency domain. Let $(\mathcal{F}v)(z)$ denote the fast Fourier transform (FFT) of a vector $v$, where $z$ is the angular frequency. To assess the accuracy of an approximate solution $Y$, we compute the relative error at each frequency $z \in \{\pi j - 10001\pi : j \in \mathbb{N}_{20001}\}$ as follows:
\begin{equation*}
    \frac{\left| \mathcal{F}([y(s_j)]_{j=1}^{20001})(z) - \mathcal{F}([Y(s_j)]_{j=1}^{20001})(z) \right|}{\left| \mathcal{F}([y(s_j)]_{j=1}^{20001})(z) \right|},
\end{equation*}
where $y$ is the exact solution, $s_j := -1 + \frac{2(j-1)}{20000}$ for $j \in \mathbb{N}_{20001}$, and $Y = \tilde{y}_l^*$ denotes the approximate solution obtained at grade $l$ in an $L$-grade AMGDL model. 

Figure \ref{fig:exp1_grade_freq2} visualizes how the AMGDL Algorithm with varying-width improves the approximation grade by grade from the frequency domain perspective.

\begin{figure}[H]
\centering
 \includegraphics[width=0.6\textwidth]{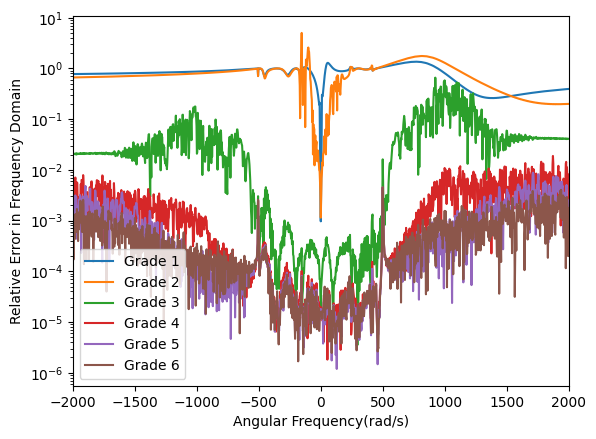}
 \caption{Frequency-domain relative errors for AMGDL  using varying-width networks: grades $1$–$8$.}
\label{fig:exp1_grade_freq2}
\end{figure}

We also compare the performance of the AMGDL model with that of the SGDL model. For each $j \in \mathbb{N}_{10}$, the SGDL-$j$ model employs a network architecture of the form:
\[
[1] \to [a_1] \to [a_2] \to \cdots \to [a_j] \to [2],
\]
where the width $a_k := 200 + 100\lceil k/2 \rceil$, for $k = 1, \dots, j$, matching the layer widths used in the AMGDL model with varying-width.

The training and hyperparameter tuning strategy is identical to that used for the equal-width SGDL models. The optimal hyperparameters and corresponding errors for the SGDL models are summarized in Table \ref{tab:exp1_sgl2}.

\begin{table}[!ht]
\centering
\scalebox{0.9}{
\begin{tabular}{c||c|c||c|c|c}
\hline
Model  & Batch size & Initial learning rate & Training loss & Validation loss & RE               \\ \hline
SGDL-1  & 512       & 1e-2                  & 3.29e-0       & 3.25e-0         & 8.86e-1          \\
SGDL-2  & 256         & 1e-1                 & 7.59e-6       &    7.52e-6      &  1.37e-3      \\ 
SGDL-3  & 256        & 1e-2                  & 4.73e-7       & 5.87e-7         & 3.73e-4          \\
SGDL-4  & 512        & 1e-2                  & 6.60e-7       & 6.25e-7         & 4.15e-4          \\
SGDL-5  & 512        & 1e-2                  & 4.72e-7       & 6.00e-7         & 4.00e-4          \\
SGDL-6  & 128        & 1e-3                  & 7.66e-7       & 5.23e-7         & 3.89e-4          \\
SGDL-7  & 128        & 1e-3                  & 2.98e-7       & {\bf 2.27e-7}         & \textbf{2.64e-4} \\
SGDL-8  & 256        & 1e-3                  & 5.01e-7       & 4.32e-7         & 3.47e-4          \\
SGDL-9  & 256       &   1e-3                 & 2.78e-7        & 2.67e-7         & 2.99e-4          \\
SGDL-10 & 256       &   1e-3                 &  1.50e-7      & 3.17e-7         &  3.87e-4          \\
SGDL-11  & 512      &  1e-3                 & 6.39e-7       & 5.67e-7       &  3.96e-4       \\
SGDL-12 & 512       &  1e-3                   & 3.41e-7         & 3.46e-7          & 3.12e-4           \\ 
\hline
\end{tabular}}
\caption{Optimal hyperparameters and corresponding errors for SGDL models using 
varying-width networks.}
\label{tab:exp1_sgl2}
\end{table}

The numerical experiments confirm the theoretical results established in the previous sections. In particular, Figures \ref{fig:exp1_grade1} and \ref{fig:exp1_grade2} show that the relative error is bounded by the training error up to a constant factor, as stated in Theorem \ref{prop_error_multi_discrete2}, assuming the quadrature error is negligible. Moreover, the last columns of Tables \ref{tab:exp1_amgdl1} and \ref{tab:exp1_amgdl2} verify that the computational time scales linearly with the number of grades, in agreement with Theorem \ref{Training_Time}.

Finally, we summarize the relative errors of all models used in this experiment in Table \ref{tab1:err_summary}. As shown in the table, for both equal-width and varying-width network architectures, AMGDL consistently outperforms SGDL by approximately one order of magnitude in accuracy. Additionally, within both the AMGDL and SGDL frameworks, varying-width networks achieve slightly better performance than their equal-width counterparts. Among all the models evaluated, our proposed V-AMGDL algorithm achieves the lowest error and delivers the best overall performance.

\begin{table}[H]
\centering
\scalebox{0.7}{
\begin{tabular}{c|cc|cc}
\hline
Model & E-AMGDL & E-SGL & V-AMGDL & V-SGL  \\ \hline\hline
RE  & 8.69e-5  & 3.37e-4 & \textbf{7.48e-5}  & 2.64e-4  \\ \hline
\end{tabular}}
\caption{Summary of relative errors (RE) for different models. “E-” and “V-” denote models with equal-width and varying-width network architectures, respectively.}
\label{tab1:err_summary}
\end{table}

\vspace{10pt}

\noindent\textbf{Example 2:} This example is designed to assess the performance of the AMGDL algorithm in learning solutions that exhibit multiple oscillatory scales, featuring a singularity and a highest wavenumber of $\kappa = 500$.

We consider the oscillatory Fredholm integral equation \eqref{fredholm_equation_operator} with the  constant kernel $K(s,t) = 0.45, s,t\in I$ and $\kappa = 500$. The constant kernel is chosen here to contrast with the previous example, where a non-separable kernel was used. The exact solution is given by
\begin{align}
    y(s) :=&\, \text{sign}(s)\,|s|^\epsilon \ln(|s|) \big[\sin(|s|) + s^3 e^{100is} + \cosh(s) e^{300is} + s^2 e^{-350is} + \label{example_solution2} \\
    &\quad \sinh(s) e^{400is} + (s^3 + \sin(s)) e^{450is} + |s|^3 e^{-500is} \big], \quad s \in I \setminus \{0\}, \nonumber
\end{align}
where $\epsilon > 0$, and $\text{sign}(s) := 1$ for $s > 0$ and $-1$ for $s < 0$. To ensure continuity, we define $y(0) := 0$, making $y$ a continuous function on the interval $I$. The corresponding right-hand side $f$ of the integral equation \eqref{fredholm_equation_operator} is then computed accordingly.

\begin{comment}
\begin{table}[H]
\centering
\begin{tabular}{|c|c|c|c|c|c|c|}
\hline
$\epsilon$ & 1 & 0.8 & 0.6 & 0.4 & 0.2 & 0.1 \\ \hline
$\|y\|_2$ estimation & 0.5717 & 0.6721 & 0.8141 & 1.0310 & 1.4032 & 1.7083 \\ \hline
\end{tabular}
\caption{The estimation of $\|y\|_2$ calculated on the test data.}
\end{table}
\end{comment}

In the following experiments, we consider
\[
\epsilon \in \{1, 0.8, 0.6, 0.4, 0.2, 0.1\}.
\]
As $\epsilon$ decreases, the solution becomes increasingly singular, posing greater challenges for numerical approximation. For the DNN-based methods, we set $\Gamma = 2$, $\gamma = 10$, $\beta = 1$, and $q = 1$, in accordance with the condition \eqref{rule_beta_gamma}. We also fix
\[
p_{500} := \lceil \gamma \cdot 500^\beta \rceil = 5,000.
\]

For the AMGDL model with equal-width networks and the corresponding SGDL model, the relative errors for different values of $\epsilon$ are presented in Tables \ref{tab:exp2_multi_result1} and \ref{tab:exp2_sgl_result1}, respectively. These results are also visualized in Figure \ref{figure2:exp2_multi_error1}. To keep the presentation concise, we report only the relative errors in this example, omitting training and validation losses due to space limitations.

\begin{table}[!ht]
\centering
\scalebox{0.9}{
\begin{tabular}{c||c|c|c|c|c|c}
\hline
 & $\epsilon=1$ & $\epsilon=0.8$ & $\epsilon=0.6$ & $\epsilon=0.4$ & $\epsilon=0.2$ & $\epsilon=0.1$ \\ \hline
Grade 1 &9.56e-1&2.87e-0& 2.58e-0 & 2.26e-0 & 1.94e-0 & 1.82e-0 \\ 
Grade 2 &9.55e-1&1.42e-0& 1.28e-0 & 1.24e-0 & 1.13e-1& 1.13e-0 \\ 
Grade 3 &1.82e-1&4.69e-2& 3.20e-2 & 3.77e-2 & 5.71e-2& 8.22e-2\\ 
Grade 4 &2.08e-3&4.73e-3& 5.01e-3 & 7.14e-3 & 1.94e-2 & 3.34e-2\\ 
Grade 5 &5.21e-4&3.66e-3& 3.89e-3 & 6.29e-3 & 1.23e-2 & 2.43e-2\\ 
Grade 6 &1.12e-4&2.55e-3& 3.36e-3 & 5.39e-3 & 1.13e-2 & 2.32e-2\\ 
Grade 7 &1.08e-4&1.96e-3& 2.64e-3 & 4.41e-3 & 1.09e-2 & 2.29e-2\\ 
Grade 8 &7.17e-5&\textbf{1.80e-3}& \textbf{2.50e-3} & \textbf{4.05e-3} & \textbf{1.07e-2} & \textbf{2.23e-2} \\ 
Grade 9 &\textbf{6.92e-5}&6.28e-3& 3.43e-3 & 6.02e-3& 1.25e-2 & 2.38e-2 \\
Grade 10 &6.98e-5&-& - & - & - & - \\ 
 \hline \hline
Best RE & 6.92e-5 & 1.80e-3 & 2.50e-3 & 4.05e-3 & 1.07e-2 & 2.23e-2 \\ \hline
\end{tabular}}
\caption{Relative errors of AMGDL using equal-width networks across different grades and singularity levels.}
\label{tab:exp2_multi_result1}
\end{table}

\begin{table}[H]
\centering
\scalebox{0.9}{
\begin{tabular}{c||c|c|c|c|c|c}
\hline
 & $\epsilon=1$ & $\epsilon=0.8$ & $\epsilon=0.6$ & $\epsilon=0.4$ & $\epsilon=0.2$ & $\epsilon=0.1$ \\ \hline
SGDL-1 & 9.56e-1 & 2.47e-0 & 2.23e-0 & 1.98e-0 & 1.78e-0 & 1.71e-0 \\
SGDL-2 & 2.46e-3 & 2.24e-2 & 1.62e-2 & 2.17e-2 & 4.11e-2 & 6.99e-2 \\
SGDL-3 & 5.38e-4 & 2.58e-2 & 1.95e-2 & 2.70e-2 & 3.52e-2 & 4.89e-2 \\
SGDL-4 & 6.01e-4 & 3.83e-2 & 3.42e-2 & 3.22e-2 & 3.27e-2 & 4.65e-2 \\
SGDL-5 & 2.42e-4 & 5.79e-2 & 3.55e-2 & 4.33e-2 & 7.52e-2 & 8.78e-1 \\
SGDL-6 & 1.07e-3 & 2.34e-1 & 3.14e-1 & 5.90e-1 & 4.73e-2 & 5.26e-2 \\
SGDL-7 & 7.00e-4 & 1.24e-1 & 1.68e-1 & 2.12e-1 & 4.51e-2 & 3.27e-2 \\
SGDL-8 & 5.40e-4 & 1.04e-2 & 5.51e-3 & 1.32e-2 & 2.26e-2 & 3.36e-2 \\
SGDL-9 & 4.84e-4 & 3.99e-3 & 5.50e-3 & 8.96e-3 & 1.37e-2 &  5.55e-2\\
SGDL-10 & 4.12e-4 & 4.44e-3 &\textbf{4.10e-3} & 8.25e-3 & \textbf{1.23e-2} & 3.75e-2 \\
SGDL-11 & \textbf{3.53e-4} &\textbf{3.98e-3}  & 4.63e-3 & \textbf{7.28e-3} & 1.26e-2 & 2.51e-2 \\
SGDL-12 & 3.93e-4 & 8.85e-3 & 5.23e-3 & 7.43e-3 & 1.67e-2 & \textbf{2.25e-2}  \\ \hline\hline
Best RE & 3.53e-4 & 3.98e-3 & 4.10e-3 & 7.28e-3  & 1.23e-2  & 2.25e-2 \\ \hline
\end{tabular}}
\caption{Relative errors for SGDL models using equal-width networks across singularity levels.}
\label{tab:exp2_sgl_result1}
\end{table}

\begin{figure}[htbp]
  \centering
  \subfigure[AMGDL]{
    \begin{minipage}{0.46\textwidth}
      \centering
      \includegraphics[width=\textwidth]{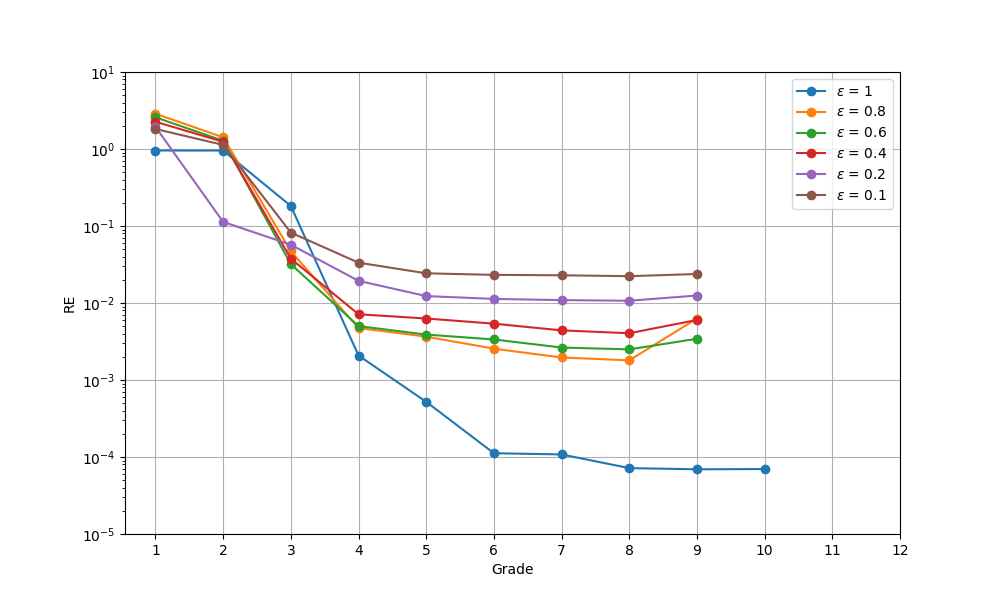}
      %\label{fig:plot1}
    \end{minipage}
  }
  \hfill
  \subfigure[SGDL]{
    \begin{minipage}{0.46\textwidth}
      \centering
      \includegraphics[width=\textwidth]{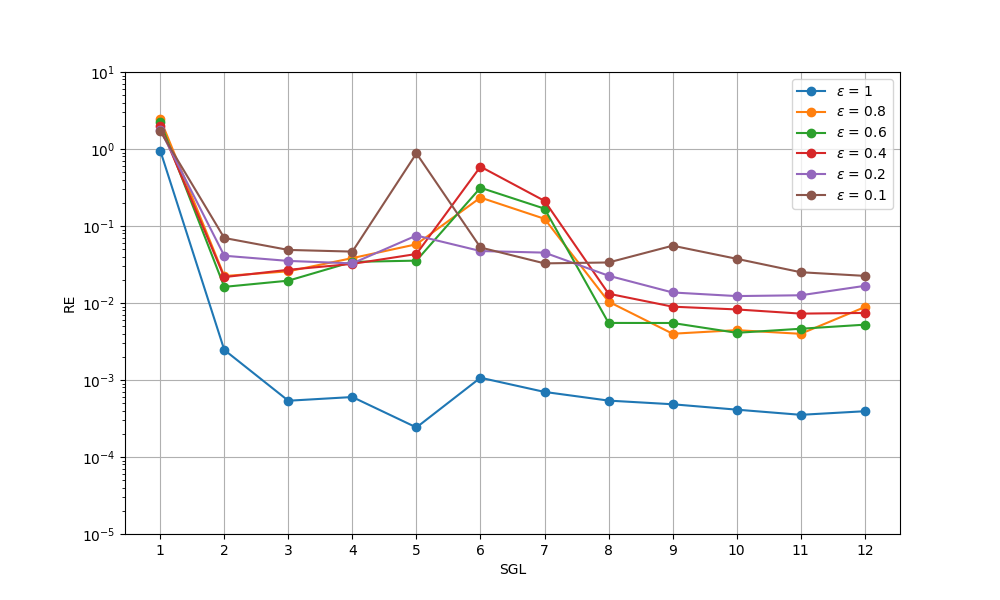}
      %\label{fig:plot2}
    \end{minipage}
  }
  \caption{Relative errors for AMGDL and SGDL models using equal-width networks across different singularity levels.}
  \label{figure2:exp2_multi_error1}
\end{figure}

% \begin{figure}[ht]
% \centering
%  \includegraphics[width=0.8\textwidth]{figure/exp2/equ_multi.png}
%      \caption{Relative error for AMGDL with equal width for different $\epsilon$.}
%      \label{figure2:exp2_multi_error1}
% \end{figure}

For the case of varying-width networks, the relative errors of the AMGDL and the corresponding SGDL models across different levels of singularity   $\epsilon$ are presented in Tables \ref{tab:exp2_multi_result2} and \ref{tab:exp2_sgl_result2}, respectively, and visualized in Figure \ref{figure2:exp2_multi_error2}.

\begin{table}[!ht]
\centering
\scalebox{0.9}{
\begin{tabular}{l||c|c|c|c|c|c}
\hline
 & $\epsilon=1$ & $\epsilon=0.8$ & $\epsilon=0.6$ & $\epsilon=0.4$ & $\epsilon=0.2$ & $\epsilon=0.1$ \\ \hline
Grade 1 & 9.56e-1 & 2.87e-0 & 2.58e-0 & 2.26e-0 & 1.95e-0 & 1.82e-0 \\ 
Grade 2 & 9.47e-1 & 1.35e-0 & 1.27e-0 & 1.15e-0 & 1.11e-0 & 1.13e-0 \\ 
Grade 3 & 9.59e-3 & 1.73e-2 & 1.25e-2 & 1.73e-2 & 4.30e-2 & 1.10e-1 \\ 
Grade 4 & 2.82e-4 & 4.75e-3 & 3.72e-3 & 6.05e-3 & 1.75e-2 & 4.30e-2 \\ 
Grade 5 & 1.08e-4 & 2.92e-3 & 2.74e-3 & 4.86e-3 & 1.37e-2 & 3.00e-2 \\ 
Grade 6 & 4.76e-5 & 2.10e-3 & 2.27e-3 & 3.89e-3 & 1.15e-2 & 2.40e-2 \\ 
Grade 7 & \textbf{3.62e-5} & 1.41e-3 & 2.09e-3 & 3.44e-3 & 1.14e-2 & \textbf{2.20e-2} \\ 
Grade 8 & 4.35e-5 & \textbf{1.03e-3} & 1.76e-3 & 3.14e-3 & \textbf{1.13e-2} & 3.06-2 \\ 
Grade 9 & - & 1.93e-2 & \textbf{1.48e-3} & 2.87e-3 & 1.25e-2 &-  \\
Grade 10 & - & - & 6.20e-2 & 2.60e-3 & - &  -\\ 
Grade 11 & - & - & - & \textbf{2.32e-3} & - & - \\ 
Grade 12 & - & - & - & 1.53e-2 & - & - \\ \hline\hline
Best & 3.62e-5 & 1.27e-3 & 1.48e-3 & 2.32e-3 & 1.13e-2 & 2.20e-2 \\ \hline
\end{tabular}}
\caption{Relative errors of AMGDL using varying-width networks across different grades and singularity levels.}

\label{tab:exp2_multi_result2}
\end{table}

\begin{table}[H]
\centering
\scalebox{0.9}{
\begin{tabular}{l||c|c|c|c|c|c}
\hline
 & $\epsilon=1$ & $\epsilon=0.8$ & $\epsilon=0.6$ & $\epsilon=0.4$ & $\epsilon=0.2$ & $\epsilon=0.1$ \\ \hline
SGDL-1 & 9.56e-1 & 2.47e-0 & 2.25e-0 & 1.95e-0 & 1.75e-0 & 1.71e-0 \\ 
SGDL-2 & 5.03e-4 & 2.36e-2 & 1.60e-2 & 1.76e-2 & 3.50e-2 & 6.23e-2 \\ 
SGDL-3 & 4.89e-4 & 1.75e-2 & 1.76e-2 & 1.86e-2 & 3.03e-2 & 4.89e-2 \\ 
SGDL-4 & 3.27e-4 & 3.01e-2 & 3.25e-2 & 3.30e-2 & 3.05e-2 & 4.14e-2 \\ 
SGDL-5 & \textbf{2.17e-4} & 1.05e-1 & 8.79e-2 & 3.36e-1 & 1.68e-1 & 6.57e-2 \\ 
SGDL-6 & 7.50e-4 & 5.85e-3 & 7.20e-3 & 1.01e-2 & 2.27e-2 & 3.77e-2 \\ 
SGDL-7 & 3.75e-4 & 3.15e-3 & 4.18e-3 & 9.25e-3 & 1.54e-2 & 4.78e-2  \\ 
SGDL-8 & 2.27e-4 & \textbf{2.45e-3} & \textbf{4.05e-3} & \textbf{4.88e-3} & \textbf{1.30e-2} & 2.83e-2 \\ 
SGDL-9 & 7.16e-4 & 3.02e-3 & 7.90e-3 & 7.58e-3 & 1.56e-2 & \textbf{2.52e-2} \\ 
SGDL-10 & 1.01e-3  & 5.77e-3  & 9.49e-3  & 7.64e-3 & 2.28e-2 & 3.30e-2 \\ 
SGDL-11 & 4.21e-4  & 8.36e-3 & 7.47e-3 & 8.57e-3 &1.42e-2  & 3.19e-2 \\ 
SGDL-12 & 3.72e-3 & 6.83e-3  & 8.36e-3 & 1.03e-2 & 1.99e-2& 3.52e-2 \\ \hline\hline
Best & 2.17e-4 & 2.45e-3 & 4.05e-3  & 4.88e-3 &1.30e-2  & 2.52e-2 \\ \hline
\end{tabular}}
\caption{Relative errors for SGDL models using varying-width networks across singularity levels.}
\label{tab:exp2_sgl_result2}
\end{table}

\begin{figure}[htbp]
  \centering
  \subfigure[AMGDL]{
    \begin{minipage}{0.46\textwidth}
      \centering
      \includegraphics[width=\textwidth]{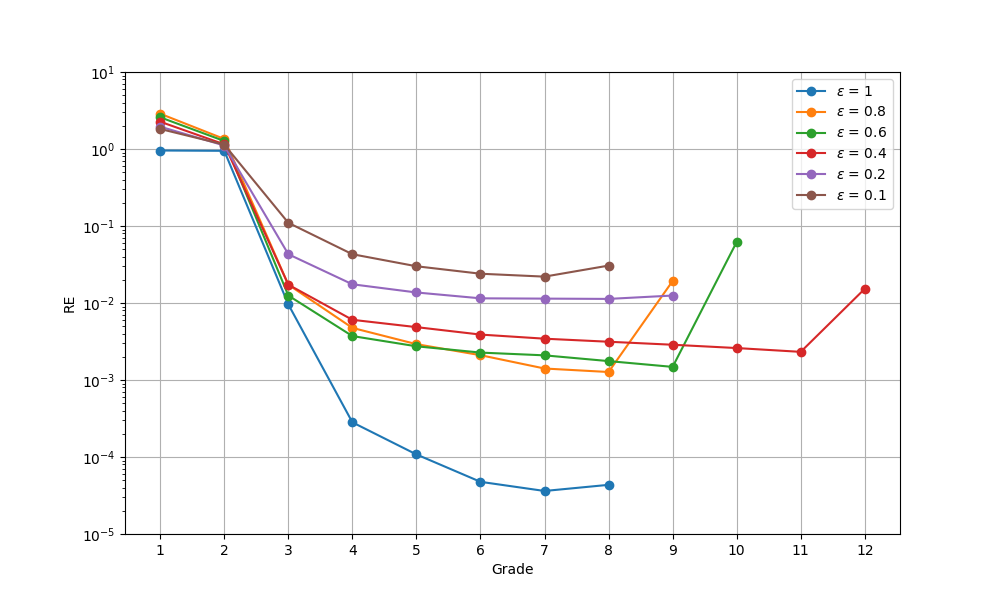}
      %\label{fig:plot1}
    \end{minipage}
  }
  \hfill
  \subfigure[SGDL]{
    \begin{minipage}{0.46\textwidth}
      \centering
      \includegraphics[width=\textwidth]{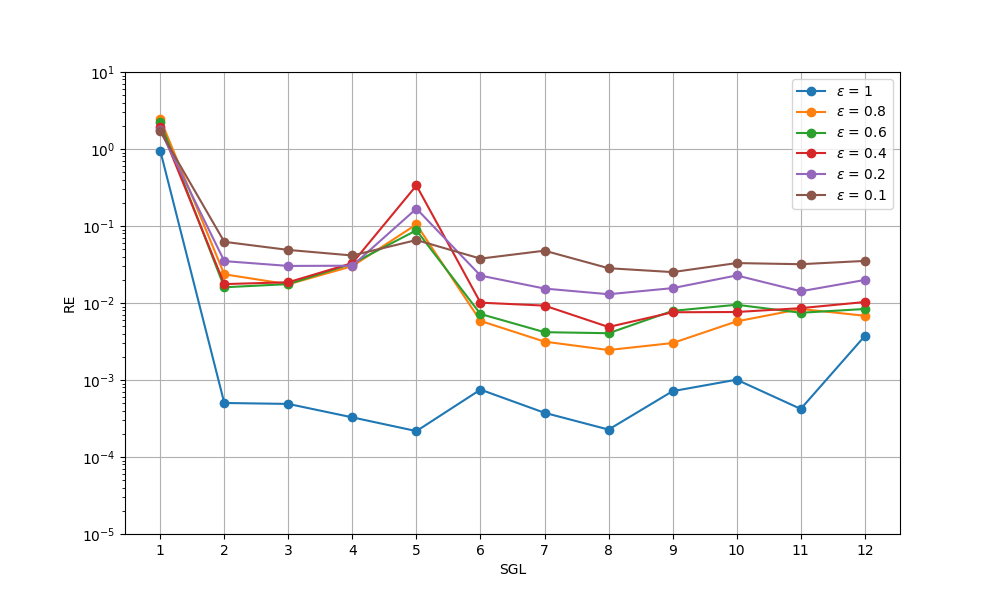}
      %\label{fig:plot2}
    \end{minipage}
  }
\caption{Relative errors for AMGDL and SGDL models using varying-width networks across different singularity levels.}

  \label{figure2:exp2_multi_error2}
\end{figure}

Both Tables \ref{tab:exp2_multi_result1} and \ref{tab:exp2_multi_result2} demonstrate that the proposed algorithm maintains robust adaptivity across different levels of singularity. A key advantage is that the depth of the neural network does not need to be predetermined; it is automatically determined by the stopping criterion. This overcomes a common limitation of traditional DNNs, where the optimal network depth for a given solution is generally unknown in advance.

Finally, we summarize the relative errors of all models evaluated in this experiment in Table \ref{tab2:err_summary}. As shown in the table, all deep learning models—both SGDL and AMGDL—achieve reasonably low error levels across various singularities. However, AMGDL models consistently outperform their single-grade counterparts, particularly as the singularity increases (i.e., smaller $\epsilon$ values). Moreover, the AMGDL models with varying-width networks yield slightly lower errors than those with equal-width architectures, demonstrating greater robustness to solution singularities. Overall, the proposed V-AMGDL method achieves the best performance, especially for highly singular problems.

The numerical results in this section demonstrate that, regardless of whether the smooth kernel 
$K$ is separable or non-separable, the AMGDL model consistently outperforms the SGDL model across both equal-width and varying-width networks, as well as for solutions exhibiting different types of singularities. Moreover, within the AMGDL framework, varying-width networks achieve better performance than their equal-width counterparts.

\begin{table}[H]
\centering
\scalebox{0.7}{
\begin{tabular}{c||c|c||c|c}
\hline
Model & E-AMGDL & E-SGL (Best) & V-AMGDL & V-SGL (Best)\\ \hline
$\epsilon=1$ & 6.92e-5 & 3.53e-4 & \textbf{3.62e-5} & 2.17e-4    \\ 
$\epsilon=0.8$ & 1.80e-3 & 3.98e-3 & \textbf{1.03e-3} & 2.45e-3   \\ 
$\epsilon=0.6$ & 2.50e-3 & 4.10e-3 & \textbf{1.48e-3} & 4.05e-3   \\ 
$\epsilon=0.4$ & 4.05e-3 & 7.28e-3 & \textbf{2.32e-3} &4.88e-3  \\ 
$\epsilon=0.2$ & \textbf{1.07e-2} & 1.23e-2 & 1.13e-2 & 1.30e-2  \\ 
$\epsilon=0.1$ & 2.23e-2 & 2.25e-2 & \textbf{2.20e-2}  & 2.52e-2    \\ \hline
\end{tabular}}
\caption{Summary of relative errors (RE) for different models across various values of $\epsilon$. “E-” and “V-” indicate models with equal-width and varying-width networks, respectively. “(Best)” denotes the best performance among 12 SGDL models.}
\label{tab2:err_summary}
\end{table}

\section{Conclusion}

This paper presents a comprehensive theoretical and algorithmic development of Multi-Grade Deep Learning (MGDL) for solving highly oscillatory Fredholm integral equations. While prior work applied MGDL empirically to this class of problems, our work is the first to provide a rigorous error analysis of both the continuous and discrete MGDL models in this context. A key theoretical contribution is the demonstration that, under sufficiently accurate quadrature, the discrete model inherits the convergence and stability properties of the continuous formulation.

Building on this analysis, we developed a novel adaptive MGDL algorithm that dynamically selects the number of network grades based on training error. This adaptivity is theoretically justified by our result showing that the dominant source of approximation error lies in the DNN training process, not the discretization or quadrature. The proposed method incrementally constructs the network grade-by-grade, eliminating the need to pre-specify network depth—a significant limitation in standard deep learning approaches.

Numerical experiments confirm the effectiveness and robustness of the adaptive MGDL algorithm, particularly in challenging settings involving high-frequency oscillations and singular solutions. Compared to conventional single-grade networks, the adaptive MGDL achieves superior accuracy and generalization.

Overall, this work introduces a theoretically grounded, scalable, and adaptive deep learning framework for operator equations, demonstrating the distinct advantages of MGDL in handling multiscale and oscillatory phenomena.

\bigskip

\noindent
{\bf Acknowledgment:} Y. Xu is a Professor Emeritus of Mathematics at Syracuse University. All correspondence should be addressed to Y. Xu.

\bigskip
\noindent
{\bf Funding Statement:} Y. Xu is partially supported by the U.S. National Science
Foundation under grant DMS-2208386.

\bigskip

\noindent
{\bf Data Availability:} The data will be made available on reasonable request.

\bigskip

\noindent
{\bf Declarations:}
The authors declare that they have no conflict of interest.

%\begin{acknowledgements}
%\end{acknowledgements}


\bibliographystyle{spmpsci}      % mathematics and physical sciences
\bibliography{references}   % name your BibTeX data base

\end{document}